\documentclass[11 pt]{amsart}
\usepackage{graphicx} 
\usepackage{amsmath}
\usepackage{mathtools}
\usepackage{comment}
\usepackage{amstext, amscd, setspace, tikz-cd,mathtools, enumerate, graphics, latexsym, siunitx}
\usepackage{pst-node}
\usepackage{tikz-cd} 
\usepackage{amssymb}
\usepackage[pagebackref]{hyperref}

\usepackage{PTSerif} 
\usepackage{graphicx}
\usepackage[cmtip,all]{xy}

\usepackage{relsize}
\usepackage[bbgreekl]{mathbbol}
\usepackage{amsfonts}

\DeclareSymbolFontAlphabet{\mathbb}{AMSb}
\DeclareSymbolFontAlphabet{\mathbbl}{bbold}

\newcommand{\Prism}{{\mathlarger{\mathbbl{\Delta}}}}

\newcommand{\cE}{{\mathcal E}}

\newcommand{\cI}{{\mathcal I }}

\newcommand{\cM}{{\mathcal M}}

\newcommand{\cO}{{\mathcal O}}

\newcommand{\cT}{{\mathcal T }}

\newcommand{\cX}{{\mathcal X }}
\newcommand{\cY}{{\mathcal Y}}

\newcommand{\fN}{\mathfrak N}

\newcommand{\fM}{\mathfrak M}
\newcommand{\fS}{\mathfrak S}

\newcommand{\Q}{{\mathbb Q}}
\newcommand{\bL}{{\mathbb L}}

\newcommand{\Z}{{\mathbb Z}}
\newcommand{\N}{{\mathbb N}}
\newcommand{\C}{{\mathbb C}}

\newcommand{\Hom}{{\rm Hom}}

\newcommand{\Gal}{{\rm Gal}}
\newcommand{\im}{{\rm image~}}

\input{xypic}
\xyoption{all}

\newtheorem{defn}[subsubsection]{{\bf Definition}}

\newtheorem{lem}[subsubsection]{{\bf Lemma}}
\newtheorem{thm}[subsubsection]{{\bf Theorem}}

\newtheorem{prop}[subsubsection]{{\bf Proposition}}
\newtheorem{rem}[subsubsection]{{\bf Remark}}
\theoremstyle{definition}

\newtheorem{exa}[subsubsection]{{\bf Example}}

\numberwithin{equation}{section}

\title[Potential semistability of Finite height Galois representations]{Potential semistability of Finite height Galois representations: Relative case}

\author{Kaustabh Mondal}

\date{\today}
\subjclass[2020]{}


\address{Indian Institute of Science Education and Research, Pune, Dr.\,Homi Bhabha Road, Pashan, Pune 411008,  INDIA.}

\email{kaustabh.mondal@students.iiserpune.ac.in}

\begin{document}
\begin{abstract}
    Let $K$ be a $p$-adic field. We define the notion of finite height for an étale $\Z_p$-local system on a smooth affinoid adic space $\cX$ over $K$ with semistable reduction. Using analytic prismatic $F$-crystals and purity results of Du-Liu-Moon-Shimizu, we prove that if an étale $\Z_p$-local system over $\cX$ is of finite height then its pullback along a finite étale cover of $\cX$ is semistable. This answers a question of Tong Liu in the relative setting.
\end{abstract}

\maketitle

\setcounter{tocdepth}{1}
\tableofcontents{}

\section{Introduction}\label{sec-1}

\subsection{Finite $E$-height $\Z_p$-local systems:}

Let $p$ be a fixed odd prime. Consider a finite extension $K$ of $\Q_p$ and its ring of integers $\cO_K$ with residue field $k$. Fix a uniformizer $\pi$ in $\cO_K$. The minimal polynomial of $\pi$ is an irreducible polynomial $E:=E(u)$ in $W(k)[u]$; where $W(k)$ is the ring of Witt vectors. Fix an algebraic closure $\overline{K}$ of $K$ and denote $\Gal_K:=\Gal(\overline{K}/K)$. A $p$-adic representation of $\Gal_K$ is a continuous representation on a finite dimensional $\Q_p$-vector space $V$; that is, a continuous homomorphism $\rho : \Gal_K \rightarrow \text{Aut}_{\Q_p}(V)$. Note that $V$ contains a $\Gal_K$-stable $\Z_p$-lattice $T$. We abusively say $T$ to be $\Z_p$-representation of $\Gal_K$. The goal of (algebraic) integral $p$-adic Hodge theory is to use certain linear algebraic objects to classify $\Z_p$-lattices in various $p$-adic representations of $\Gal_K$ such as crystalline, semistable, de Rham representations, etc. These properties are defined by admissibility formalism using Fontaine's period rings $B_{\text{cris}}, B_{\text{st}}, B_{\text{dR}}$ etc. (see e.g. \cite{Fon94b})\smallskip

\noindent Another class of $p$-adic representations of $\Gal_K$, called \textit{finite $E$-height} representations,  has been studied by T. Liu in \cite{Liu}. We first describe some linear algebraic objects to define finite $E$-height Galois representations. Let $\fS_K=W(k)[[u]]$. It is equipped with Frobenius $\varphi$ defined by $\varphi(u)=u^p$ and the natural Frobenius on $W(k)$. 

\begin{defn}
A finite height Breuil-Kisin module over $\fS_K$ is a pair $(\fM_K, \varphi_{\fM_K})$; where $\fM_K$ is a finite free $\fS_K$-module and  $\varphi_{\fM_K}$ is a $\varphi$-semilinear endomorphism on $\fM_K$ such that the cokernel of the linearization $\varphi_{\fM_K}^\star=\mathrm{id} \otimes \varphi_{\fM_K} : \fS_K \otimes_{\varphi, \fS} \fM_K \rightarrow \fM_K$ is killed by some non-negative power of $E(u)$ i.e. $E^r  \fM_K \subset \text{Image}(\varphi_{\fM_K}^\star)$. More specifically, we call $(\fM_K, \varphi_{\fM_K})$ a Breuil-Kisin module of $E$-height $\leq r$. We denote the category of Breuil-Kisin modules of $E$-height $\leq r$ by $\textbf{Mod}^{r, fr}_{/ \fS_K}$.  
\end{defn}

\noindent Let us write $\C_p=\widehat{\overline{K}}$ and let $\cO_{\C_p}$ be its ring of integers. Let $\C_p^\flat$ be the tilt of $\C_p$ and $\cO_{\C_p^\flat}$ be its ring of integers. We fix once and for all a $p$-power compatible sequence $\{ \pi_n:= \pi^{1/p^n}\}$ of $\pi$ inside $\cO_{\C_p}$. This gives an element $\pi^\flat \in \cO_{\C_p^\flat}$. We write $A_{\text{inf}} := W(\cO_{\C_p^\flat})$. Thus, we have the constant lift $[\pi^\flat] \in A_{\text{inf}}$. We embed $\fS \hookrightarrow A_{\text{inf}}$ by $u \mapsto [\pi^\flat]$, which is compatible with Frobenius endomorphisms. The fixed $p$-power compatible sequence $\{\pi_n\}$ gives rise to an arithmetic Kummer tower $K_{\infty}=\cup_n K(\pi_n)$. Let $\Gal_{K_{\infty}}=\Gal(\overline{K} / K_{\infty})$. To every Breuil-Kisin module $(\fM_K, \varphi_{\fM_K})$ in $\textbf{Mod}^{r, fr}_{/ \fS_K}$, we associate a $\Z_p[\Gal_{K_{\infty}}]$-module $T_{\fS_K}(\fM_K):=\Hom_{\fS_K, \varphi}(\fM_K, W(\C_p^\flat))$. It is easy to see that $T_{\fS_K}(\fM_K)$ is a finite free $\Z_p$-module and $\text{Rank}_{\Z_p}(T_{\fS_K}(\fM_K))=\text{Rank}_{\fS_K}(\fM_K)$. Now we are ready to define the finite $E$-height $p$-adic representation of $\Gal_K$.

\begin{defn}\label{fin-defn-classical}
A $p$-adic representation $V$ of $\Gal_K$ is said to be of finite $E$-height (w.r.t. our fixed choice $\overrightarrow{\pi}=\{\pi_n\}$) if there exists a $\Gal_K$-stable $\Z_p$-lattice $T$ inside $V$ and a finite height Breuil-Kisin module $(\fM_K, \varphi_{\fM_K})$ such that $T|_{\Gal_{K_{\infty}}} \cong T_{\fS_K} (\fM_K)$. We sometimes abusively call $T$ to be a finite height $\Z_p$-representation of $\Gal_K$.
\end{defn}

\begin{rem}
The above definition is Def. 1.1.15 of \cite{Gao}. In \cite{Liu}, the finite $E$-height $p$-adic representation is defined by the existence of $\Gal_{K_{\infty}}$-stable $\Z_p$-lattice $T$ and finite $E$-height Breuil-Kisin module $(\fM_K, \varphi_{\fM_K})$ such that $T|_{\Gal_{K_{\infty}}} \cong T_{\fS_K} (\fM_K)$. However, by the proof of Lem. 2.1.15 in \cite{Kis}, if $V$ is of finite $E$-height representation (in the sense of Def. \ref{fin-defn-classical}) then any $\Gal_{K_{\infty}}$-stable $\Z_p$-lattice $T$ in $V$ arises from a finite height Breuil-Kisin module $\fN_K \in \textbf{Mod}^{r, fr}_{/ \fS_K}$. In particular, the two definitions are equivalent (see Rem. 2.1.2 in \cite{Liu}). 
\end{rem}

Liu asked the Ques. 4.3.1 (2) in \cite{Liu} that whether a finite $E$-height $p$-adic representation $V$ of $\Gal_K$ is potentially semistable. It has been answered by H. Gao in \cite{Gao} as the following result. 

\begin{thm}(Thm. 1.1.16 of \cite{Gao})\label{Gao-main-perf}
Let $K^{\text{ur}} \subset \overline{K}$ be the maximal unramified extension of $K$ and let $m:= \text{max}\ \{ i \in \N \cup \{0\}: \text{a primitive $p^i$-th root of unity}\ \zeta_{p^i} \in K^{\text{ur}} \}$. Write $K_m = K(\pi_m)$. Let $T$ be a finite free $\Gal_K$-stable $\Z_p$-lattice in a $p$-adic representation $V$ of $\Gal_K$. Then $T$ is a finite $E$-height $\Z_p$-representation of $\Gal_K$ if and only if $V|_{\Gal_{K_m}}$ can be extended to a semistable $\Gal_K$ representation with non-negative Hodge–Tate weights. In particular, a finite $E$-height $p$-adic representation of $\Gal_K$ is potentially semistable.
\end{thm}

The purpose of this paper is to answer  a relative version of Liu's question. Before stating our main result (Thm. \ref{main-thm-1}), let us formulate finite $E$-height property in our relative setting. We follow \cite{DLMS}, \cite{DLMSII} to define relative aspects of semistable representations.\smallskip

\noindent We work with a smooth adic space $\cX$ over $K$ with semistable reduction. Let us restrict ourselves to a small affine case. Let $X=\mathrm{Spf}(R)$, where $R$ is connected $\cO_K$-algebra equipped with a $p$-adically complete étale map $\square : R^0 \rightarrow R$ with $$ R^0 = \cO_K \langle X_1, X_2, \cdots, X_r, X_{r+1}^{\pm}, X_{r+2}^{\pm}, \cdots X_{b}^{\pm} \rangle / (X_1.X_2. \dots X_r - \pi).$$
$X$ is a semistable affine formal scheme over $\cO_K$. The generic fibre of $X$ gives a smooth locally noetherian affinoid adic space $\cX:=\mathrm{Spa}(R[1/p],R)$ over $K$.

\begin{defn}
Let $X_{\eta, \text{proét}}$ be the proétale site over $X_{\eta}$ in the sense of \cite{Sch13}. A $\Z_p$-local system on $\cX$ is a $p$-torsion free lisse $\widehat{\Z_p}$-sheaf on $X_{\eta, \text{proét}}$. The category of $\Z_p$-local systems on $\cX$ is denoted by $\textbf{Loc}_{\Z_p}(\cX)$.
\end{defn}

\noindent It is well-known (e.g by Thm. 4.2.3 of \cite{DL}) that the category of $\Z_p$-representations of $\Gal_K$ is equivalent to the category of Caruso's $(\varphi, \tau)$-modules (see \cite{Car}). In recent years, prismatic cohomology and prismatic $F$-crystals, developed by Bhatt and Scholze in \cite{BS22} and \cite{BS23}, revolutionized the integral $p$-adic Hodge theory significantly. One of the many advancements of the prismatic theory is that it is \textit{motivic} in nature, i.e. it recovers almost all earlier linear algebra objects to classify the $\Z_p$-lattices in $p$-adic Galois representations. Another significance is that it naturally extends to the relative setting, allowing one to study geometric families of $p$-adic Galois representations. More precisely, there is a category of objects called \textit{Laurent $F$-crystals} on absolute log prismatic site associated to $X$ which turns out to be equivalent to $\textbf{Loc}_{\Z_p}(\cX)$. Let us describe it in greater details.\smallskip

\noindent We have the semistable affine formal scheme $X=\mathrm{Spf} R$ over $\cO_K$ that admits an étale map given by $\square : R^0 \rightarrow R$. We further assume that the induced map $\square$ gives a bijection between the generic points of $X^0=\mathrm{Spf}(R^0)$ and $X=\mathrm{Spf}(R)$. This can be obtained by refining $X$ by étale covering and localizing $R^0$ accordingly. Consider the prelog structure given by $\N^r \rightarrow R^0 : e_\ell \mapsto X_\ell (\ell=1,2, \cdots, b)$. Let $M_X$ be a log structure on $X$ associated to $\N^r \rightarrow R^0 \xrightarrow{\square} R$. This gives rise to a semistable log formal scheme $(X, M_X)$. We have the absolute log prismatic site $(X, M_X)_{\Prism}$ by the log prismatic theory developed by Koshikawa in \cite{KosI}. Let us give an example of an object in $(X, M_X)_{\Prism}$ which will be used throughout the paper. 

\begin{exa}\label{example}
Let $X^0:=\mathrm{Spf}(R^0)$. Consider $$\fS_{R^0}=W(k) \langle X_1, X_2, \cdots, X_r, X_{r+1}^{\pm}, X_{r+2}^{\pm}, \cdots X_{b}^{\pm} \rangle [[u]]/ (X_1.X_2. \dots X_r - u)$$
equipped with Frobenius $\varphi(X_{\ell})=X_{\ell}^p$ for $1 \leq \ell \leq b$. $E(u) \in W(k)[u]$ defines a Cartier divisor in $\mathrm{Spec}(\fS_{R^0})$. Then, $(\fS_{R^0}, E(u)) \in X_{\Prism}$ with the structure map $\fS_{R^0} / E(u) \cong R^0$. The prelog structure $\N^r \rightarrow \fS_{R^0} : e_\ell \mapsto X_\ell$ further gives rise to log prism in $(X^0, M_{X^0})_{\Prism}$. We denote it by $(\fS_{R^0}, E, \N^r)^a$.\smallskip

\noindent Now, the $p$-adically completed étale map $\square : R^0 \rightarrow R$ lifts uniquely to $(p, E)$-completed étale map $$ \square_{\fS} : \fS_{R^0} \rightarrow \fS_{R, \square}.$$ We simply write $\fS := \fS_{R, \square}$. The log prism associated to the prelog prism structure (with respect to $\pi$ and $\square_{\fS}$) $(\mathrm{Spf} (\fS), (E(u)), \N^r \rightarrow \fS_{R^0} \xrightarrow{\square_{\fS}} \fS)$ (see Lem. 2.13 of \cite{KosI}) is denoted by $(\mathrm{Spf} (\fS), (E(u)), \N^r)^a$ as an object of $(X, M_X)_{\Prism}$ and called the Breuil-Kisin log prism.
\end{exa}
   
\noindent Now we define the category of Laurent $F$-crystals on $(X, M_X)_{\Prism}$. 

\begin{defn}
A Laurent $F$-crystal on $(X, M_X)_{\Prism}$ consists of a pair $(\cE, \varphi_{\cE})$; where $\cE$ is a crystal in vector bundles on the ringed site $((X, M_X)_{\Prism}, \cO_{\Prism}[\cI_{\Prism}^{-1}]^\wedge_p)$ and $\varphi_{\cE} : \varphi^\star \cE \rightarrow \cE$ is an isomorphism. The category of Laurent $F$-crystals on $(X, M_X)_{\Prism}$ is denoted by $\text{Vect}((X, M_X)_{\Prism}, \cO_{\Prism}[\cI_{\Prism}^{-1}]^\wedge_p)^{\varphi =1}$. 
\end{defn}

\noindent We will use some simpler linear algebraic objects (regarded as prismatic $(\varphi, \tau)$-modules) which are equivalent to Laurent $F$-crystals on $(X, M_X)_{\Prism}$. To describe it, let us look at the self-coproduct of Breuil-Kisin log prism from Exam. \ref{example}.\smallskip

\noindent The Breuil-Kisin log prism is given by the object $(\fS, E, M_{\mathrm{Spf}(\fS)})$ in $(X, M_X)_{\Prism}^{\text{opp}}$. Let $(\fS^{(i)}, E, M_{\mathrm{Spf}(\fS^{(i)})})$ be its self-coproduct and self-tripleproduct of $(\fS, E, M_{\mathrm{Spf}(\fS)})$ in $(X, M_X)_{\Prism}^{\text{opp}}$ for $i=1,2$; respectively (See Subsec. 2.3 in \cite{DLMSII} for precise description of $\fS^{(i)}$).

\begin{defn}
A pair $(\cM, \varphi_{\cM})$ is called an étale $\varphi$-module over $\fS[E^{-1}]^\wedge_p$ consists of a finite locally free $\fS[E^{-1}]^\wedge_p$-module equipped with a $\varphi$-semilinear endomorphism $\varphi_{\cM}$ on $\cM$ whose linearization $\varphi_{\cM}^\star : \varphi^\star \cM \rightarrow \cM$ is an isomorphism of $\fS[E^{-1}]^\wedge_p$-modules. We sometimes write $\cO_{\cE} := \fS[E^{-1}]^\wedge_p$ and denote the category of étale $\varphi$-modules over $\cO_{\cE}$ by $\textbf{Mod}^{\varphi}_{\cO_{\cE}}$. 
\end{defn}

\begin{prop}(Lem. 3.21 in \cite{DLMSII})\label{DD-OE}
The evaluation of Laurent $F$-crystals on the diagram $p_1, p_2 : (\fS, E, M_{\mathrm{Spf}(\fS)}) \rightarrow (\fS^{(1)}, E, M_{\mathrm{Spf}(\fS^{(1)})})$ gives an equivalence between $\text{Vect}((X, M_X)_{\Prism}, \cO_{\Prism}[\cI_{\Prism}^{-1}]^\wedge_p)^{\varphi =1}$ and the category $\textbf{DD}_{\cO_{\cE}}$ of descent data over $\fS[E^{-1}]^\wedge_p$ consists of the triplets $(\cM, \varphi_{\cM}, f_{\text{ét}})$; where $(\cM, \varphi_{\cM})$ is an étale $\varphi$-module over $\fS[E^{-1}]^\wedge_p$ and an isomorphism $$ f_{\text{ét}} : \fS^{(1)}[E^{-1}]^\wedge_p \otimes_{p_1, \fS[E^{-1}]^\wedge_p} \cM \rightarrow \fS^{(1)}[E^{-1}]^\wedge_p \otimes_{p_2, \fS[E^{-1}]^\wedge_p} \cM$$ compatible with Frobenii and satisfying cocycle conditions over $\fS^{(2)}[E^{-1}]^\wedge_p$.
\end{prop}

\noindent Both the category of Laurent $F$-crystals and the category of descent data over $\fS[E^{-1}]^\wedge_p$ are equivalent to $\textbf{Loc}_{\Z_p}(\cX)$. We summarize these equivalences in the following diagram, using Thm. 3.14 and Lem. 3.21 in \cite{DLMSII}.\smallskip

\begin{equation}\label{triangle-diag}
\begin{tikzcd}
\text{Vect}((X, M_X)_{\Prism}, \cO_{\Prism}[\cI_{\Prism}^{-1}]^\wedge_p)^{\varphi =1} \arrow[rd] \arrow[rr] &   & \textbf{DD}_{\cO_{\cE}} \arrow[ld] \\
                        & \textbf{Loc}_{\Z_p}(\cX) &             
\end{tikzcd}
\end{equation}

\begin{defn}\label{B-K-modules-intro}
A pair $(\fM, \varphi_{\fM})$ is called a Breuil-Kisin module over $\fS$ of $E$-height $\leq r$ if $\fM$ is a finite torsion free $\fS$-module such that $\fM[p^{-1}]$ (resp. $\fM[E^{-1}]$) is projective $\fS[p^{-1}]$-module (resp. $\fS[E^{-1}]$-module), $\fM = \fM[p^{-1}] \cap \fM[E^{-1}]$ and $\varphi_{\fM}$ is $\varphi$-semilinear endomorphism such that the cokernel of its linearization $\varphi_{\fM}^\star$ is killed by $E^r$ for some non-negative integer $r$.
\end{defn}

\noindent Now we define the finite $E$-height $\Z_p$-local system as follows :

\begin{defn}\label{finiteEheightlocal system-def}
A $\Z_p$-local system $\bL \in \textbf{Loc}_{\Z_p}(\cX)$ is called of finite $E$-height if the étale $\varphi$-module $(\cM, \varphi_{\cM})$ of the associated descent data $(\cM, \varphi_{\cM}, f_{\text{ét}}) \in \textbf{DD}_{\cO_{\cE}}$ as per above equivalence, arises from a Breuil-Kisin module $(\fM, \varphi_{\fM})$ over $\fS$ of finite $E$-height $\leq r$ for some non-negative $r$ i.e. $\cM = \cO_{\cE} \otimes_{\fS} \fM$.
\end{defn}

\noindent The notion of finite $E$-height $\Z_p$-local system $\bL$ depends on the choice of $p$-power compatible sequence $\{ X_i^{1/p^n}\}$ for each $1 \leq i \leq b$. Let $\overline{R}$ be the union of finite $R$-algebras $R'$ inside a fixed algebraic closure of $\text{Frac}(R)$ such that $R'[1/p]$ is étale over $R[1/p]$. Denote $\Gal_R : = \Gal (\overline{R}[1/p] / R[1/p])$. The category $\textbf{Rep}_{\Z_p}(\Gal_R)$ of finite free $\Z_p$-representation of $\Gal_R$ is equivalent to $\textbf{Loc}_{\Z_p}(\cX)$ (cf. \cite{KLI}, \cite{Sch22}). Let $\overline{R}^\wedge$ be the $p$-adic completion of $\overline{R}$ and $\overline{R}^\flat$ be its tilt. We denote $A_{\text{inf}}(\overline{R})=W(\overline{R}^\flat)$. The equivalence between $\text{Vect}((X, M_X)_{\Prism}, \cO_{\Prism}[\cI_{\Prism}^{-1}]^\wedge_p)^{\varphi =1}$ and $\textbf{Rep}_{\Z_p}(\Gal_R)$ can be described via the evaluation at the prism $(A_{\text{inf}}(\overline{R}), E, \N^r)^a$. Let $(\cE, \varphi_{\cE})$ be a Laurent $F$-crystal over $(X, M_X)_{\Prism}$ and $T$ be the corresponding $\Z_p$-representation of $\Gal_R$. The embedding $\fS \hookrightarrow A_{\text{inf}}(\overline{R})$ \textit{defined by} $X_i \mapsto [X_i^\flat]$ for $1 \leq i \leq b$ induces the log structure on $(A_{\text{inf}}(\overline{R}), E)$. Thus, we have an object $(A_{\text{inf}}(\overline{R}), E, \N^r)^a$ in $(X, M_X)^{\text{opp}}_{\Prism}$. Note that $\Gal_R$ naturally acts on the prism $(A_{\text{inf}}(\overline{R}), E, \N^r)^a$ and therefore, by functoriality, we have the action of $\Gal_R$ on the $\Z_p$-module $T^\vee=\cE((A_{\text{inf}}(\overline{R}), E, \N^r)^a)^{\varphi=1}$. On the other hand, when we evaluate $\cE$ at the Breuil-Kisin prism and its co-products, we get the descent data $(\cM, \varphi_{\cM}, f_{\text{ét}})$ over $\fS[E^{-1}]^\wedge_p$ and $T^\vee=(\cM \otimes_{\fS[E^{-1}]^\wedge_p} W(\overline{R}^\flat[(\pi^\flat)^{-1}]))^{\varphi=1}$. Therefore, if $T$ is of finite $E$-height $\Z_p$-representation of $\Gal_R$ then $T^\vee=(\fM \otimes_{\fS} W(\overline{R}^\flat[(\pi^\flat)^{-1}]))^{\varphi=1}$. This expression depends on the embedding $\fS \hookrightarrow A_{\text{inf}}(\overline{R})$, i.e. the choice of compatible system $\{ X_i^{1/p^n} \}_n$ for $1 \leq i \leq b$. In particular, $\textbf{DD}_{\cO_{\cE}}$ is appropriate replacement of the category of $(\varphi, \tau)$-modules in our relative case.

\subsection{Semistable $\Z_p$-local systems:}\label{subsec-1.2}

Let $\cX$ be the adic generic fibre of $X$. So, $X$ is a semistable formal model of $\cX$. The semistable $\Z_p$-local system has been studied extensively in \cite{Fal89}, \cite{Fal02}, \cite{AI12}, \cite{Tsu}. A semistable $\Z_p$-local system can be defined in terms of association to (filtered) $F$-isocrystals on the crystalline site of the mod-$p$ fibre $(X_1, M_{X_1})$ of $(X, M_X)$. This definition is much involved and indeed requires the detailed description of absolute crystalline site $(X_1, M_{X_1})_{\text{cris}}$, $F$-isocrystal on $(X_1, M_{X_1})_{\text{cris}}$ etc. We refer the reader to appendix B of \cite{DLMSII} for it. We follow the notion of prismatic semistability for $\Z_p$-local system as in \cite{DLMSII}. By Thm. 1.8 of \textit{loc. cit.} these notions of semistability are equivalent.

\begin{defn}
Let $(A, I, M_{\mathrm{Spf}(A)}) \in (X, M_X)_{\Prism}$. Denote $\text{Vect}^{\text{an}, \varphi}(A, I)$ to be the category of pairs $(\cE_A, \varphi_{\cE_A})$; where $\cE_A$ is a vector bundle over $\mathrm{Spec}(A) \setminus V(p, I)$ and $\varphi_{\cE_A}$ is a $\varphi$-semilinear endomorphism on $\cE_A$ which induces an isomorphism of vector bundles $ \varphi_{\cE_A}^\star : \varphi_A^\star(\cE_A)[I^{-1}] \rightarrow \cE_A[I^{-1}]$. Such an pair is called effective if the isomorphism comes from a morphism $\varphi_A^\star(\cE_A) \rightarrow \cE_A$ of vector bundles over $\mathrm{Spec}(A) \setminus V(p, I)$.\smallskip

Now, the category of analytic prismatic $F$-crystals over $(X, M_X)_{\Prism}$ is defined by $$\text{Vect}^{\text{an}, \varphi}((X, M_X)_{\Prism}):=\lim_{(\mathrm{Spf} A, I, M_{\mathrm{Spf} (A)})\in (X, M_X)_{\Prism}} \text{Vect}^{\text{an}, \varphi}(A, I).$$
An analytic prismatic $F$-crystal is denoted by $(\cE_{\Prism}, \varphi_{\cE_{\Prism}})$ or $\cE_{\Prism}$. We write $(\cE_{\Prism, A}, \varphi_{\cE_{\Prism, A}})$ for the associated object in $\text{Vect}^{\text{an}, \varphi}(A, I)$. In particular, we set $\cE_{\Prism}(A):= \cE_{\Prism}(A, I, M_{\mathrm{Spf}(A)}) = H^0(\mathrm{Spec}(A) \setminus V(p, I) , \cE_{\Prism, A})$. Finally, $(\cE_{\Prism}, \varphi_{\cE_{\Prism}})$ is called effective if $(\cE_{\Prism, A}, \varphi_{\cE_{\Prism, A}})$ is effective for each $(A, I, M_{\mathrm{Spf}(A)})$.
\end{defn}

\noindent Similar to the Laurent $F$-crystals, the analytic prismatic $F$-crystals can be described in terms of \textit{integral} Kisin descent data via evaluation at Breuil-Kisin log prism and its co-products.

\begin{defn}
$\textbf{DD}_{\fS}$ denotes the category of triplets $(\fM, \varphi_{\fM}, f)$ such that 

\begin{itemize}
\item $(\fM, \varphi_{\fM})$ is finite torsion free $\fS$-module such that $\fM[p^{-1}]$ (resp. $\fM[E^{-1}]$) is projective $\fS[p^{-1}]$-module (resp. $\fS[E^{-1}]$-module) and $\fM = \fM[p^{-1}] \cap \fM[E^{-1}]$.

\item $\varphi_{\fM}$ is $\varphi$-semilinear endomorphism such that the cokernel of its linearization $\varphi_{\fM}^\star$ is killed by $E^r$ for some non-negative integer $r$.

\item There is an isomorphism over $\fS^{(1)}$ $$f : \fS^{(1)} \otimes_{p_1, \fS} \fM \rightarrow \fS^{(1)} \otimes_{p_2, \fS} \fM$$
compatible with Frobenii and satisfying cocycle conditions over $\fS^{(2)}$.
\end{itemize}
The morphisms are $\fS$-linear maps compatible with all structures.
\end{defn} 

\begin{rem}
In the case of complete discrete valuation rings (e.g. see Sec. \ref{sec-2}), the Breuil-Kisin module $(\fM, \varphi_{\fM})$ becomes finite free $\fS$-module (cf. Cor 3.9 in \cite{DLMSII}). 
\end{rem}

\noindent The restriction along $\mathrm{Spec}(\fS) \setminus V(p,E) \hookrightarrow \mathrm{Spec}(\fS)$ provides the following equivalence (see Lem. 3.8 in \cite{DLMSII}) $$\text{Vect}^{\text{an}, \varphi}((X, M_X)_{\Prism}) \rightarrow \textbf{DD}_{\fS}\ \text{given by}\ \cE_{\Prism} \mapsto \fM=H^0(\mathrm{Spec}(\fS) \setminus V(p,E), \cE_{\Prism})$$

The definition of prismatic semistable $\Z_p$-local system is described in terms of an étale realization functor defined as follows.

\begin{defn}\label{etale-realization}
For any log prism $(\mathrm{Spf}(A), I, M_{\mathrm{Spf}(A)}) \in (X, M_X)_{\Prism}$, restriction along $\mathrm{Spec}(A) \setminus V(I) \hookrightarrow \mathrm{Spec}(A) \setminus V(p, I)$ and taking scalar extension along $A[I^{-1}] \rightarrow A[I^{-1}]^\wedge_p$ defines a functor $\text{Vect}^{\text{an}, \varphi}(A, I) \longrightarrow \text{Vect}(A[I^{-1}]^\wedge_p)^{\varphi =1}$. This induces a functor as in Lem. 3.13 of \cite{DLMSII} 
$$ T_X : \text{Vect}^{\text{an}, \varphi}((X, M_X)_{\Prism}) \rightarrow \text{Vect}((X, M_X)_{\Prism}, \cO_{\Prism}[\cI_{\Prism}^{-1}]^\wedge_p)^{\varphi =1} \cong \textbf{Loc}_{\Z_p}(\cX).$$
\end{defn} 

\noindent This is called the étale realization functor. Moreover, $T_X$ is fully faithful functor by Prop. 3.20 of \cite{DLMSII}. Recently in \cite{Ino25}, the same fully faithful étale realization functor has been established for small affine $X$ having both \text{vertical and horizontal} logarithmic structure. Now we define prismatic semistable $\Z_p$-local system as follows.

\begin{defn}
A $\Z_p$-local system $\bL \in \textbf{Loc}_{\Z_p}(\cX)$ is prismatic semistable if it lies in the essential image of the above étale realization functor $T_X$. We use this notion of semistability for a $\Z_p$-local system throughout this paper and simply call it semistable $\Z_p$-local system.
\end{defn}

\begin{rem}
Note that the definition of prismatic semistability depends on our fixed semistable model $X$ of $\cX$. In this sense, it should be called prismatic $X$-semistable. However, by Cor. 5.5 in \cite{DLMS}, $\bL \in \textbf{Loc}_{\Z_p}(\cX)$ being semistable is independent of the choice of semistable integral model $X$ of $\cX$.
\end{rem}

Now we state our main result:

\begin{thm}\label{main-thm-1}
Let $m=\text{max} \{i \in \N \cup \{0\}: \zeta_{p^i} \in \widehat{K^{\text{un}}} \}$. Let $X^0_m := \mathrm{Spf}(R_m^0)$; where $$ R_m^0 := \cO_{K_m} \langle X_1, X_2, \cdots, X_r, X_{r+1}^{\pm}, X_{r+2}^{\pm}, \cdots X_{b}^{\pm} \rangle / (X_1.X_2. \dots X_r - \pi_m).$$ 
Consider the map $\iota^0_m : X_m^0 \rightarrow X^0$ given by $X_j \mapsto X_j^{p^m}$ for $j = \{ 1,2, \cdots, r\}$ and $X_i \mapsto X_i$ for $r+1 \leq i \leq b$. Write $X_m:= X_m^0 \times_{X^0} X$. The structure map $\iota_m : X_m \rightarrow X$ induces the finite étale map $\iota_m : \cX_m \rightarrow \cX$ of their adic generic fibres. If $\bL$ is a finite $E$-height $\Z_p$-local system over $\cX$\ , then its pullback $\iota_m^\star \bL$ along the finite étale map $\iota_m : \cX_m \rightarrow \cX$ is (prismatic) semistable $\Z_p$-local system over $\cX_m$\ .
\end{thm}

\noindent The notion of potential semistability for a $\Z_p$-local system is well-known in relative $p$-adic Hodge theory. For instance, various versions of relative $p$-adic monodromy theorem state, roughly, that a de Rham $\Z_p$-local system $\bL$ on $\cX$ is potentially semistable, i.e. there exists a finite étale cover $\iota : \cY \rightarrow \cX$ of $\cX$ such that the pullback $\iota^\star \bL$ is semistable (see Rem. 2 in \cite{DMY26}; see also \cite{Du26}). Gao's result (Thm. \ref{Gao-main-perf}) shows the semistability of finite $E$-height representation at an extension $\mathrm{Spa}(K_m, \cO_{K_m})$ of $\mathrm{Spa}(K, \cO_{K})$, which is finite étale. Similarly, Thm. \ref{main-thm-1} gives us a finite étale cover $\cX_m \rightarrow \cX$ along which the pullback of a finite $E$-height $\Z_p$-local system becomes semistable.\smallskip

\noindent Note that our Thm. \ref{main-thm-1} does not fully establish the relative version of Thm. \ref{Gao-main-perf}. More precisely, it is expected that the $p$-adic local system associated to the  semistable $\Z_p$-local system $\iota_m^\star \bL$ can be \textit{extended} to a semistable $p$-adic local system (up to isogeny) over $\cX$. This would be discussed in a sequel of the current project.

\subsection{Relation to other works in literature:}\label{sec-1-3}
It is evident to the reader that this work is mostly inspired by the results of Gao in \cite{Gao}, \cite{Gao-imp}. In \textit{loc. cit.} Gao defined the category of Breuil-Kisin $\Gal_K$-modules (resp. weak Breuil-Kisin $\Gal_K$-modules) to classify the semistable $\Z_p$-representations of $\Gal_K$ (resp. finite $E$-height $\Z_p$-representations of $\Gal_K$). These objects are in a similar vein to Liu's earlier work on $(\varphi, \widehat{G})$-modules and weak $(\varphi, \widehat{G})$-modules. In \cite{Liu}, the category of $(\varphi, \widehat{G})$-modules was defined and proved to be equivalent with semistable $\Z_p$-representations of $\Gal_K$. In \cite{Oz}, Ozeki proved that the category of weak $(\varphi, \widehat{G})$-modules is equivalent to the finite $E$-height $\Z_p$-representations of $\Gal_K$ answering Ques. 4.3.1 (1) in \cite{Liu}. One can encounter a limitation while working with some analogue of $(\varphi, \widehat{G})$ in relative setting. It is defined in terms of a ring $\widehat{\mathcal{R}}$ which is not known to be $p$-adically complete (see Rem. 3.3.5 in \cite{DL} for more details). On the other hand, Gao's Breuil-Kisin $\Gal_K$-module in \cite{Gao} is defined in terms of $A_{\text{inf}}$. So, one can expect some relative version of it. In fact, in Subsec. 4.4 of \cite{Gao-imp}, Gao defined a category of Breuil-Kisin $\Gal_L$-modules for a CVDF $L$ with imperfect residue field having finite $p$-basis and proved a fully faithful functor from the category of semistable $\Z_p$-representations of $\Gal_L$ to the category of Breuil-Kisin $\Gal_L$-modules\footnote{H. Gao informed the author that Subsection 4.4 of \cite{Gao-imp} contains some gaps. However, all the results preceding that subsection are correct and some of them have been used in the present paper.}. However, it is not known to be equivalent (see Rem. 4.4.5 in \cite{Gao-imp}).\smallskip

\noindent The work of Du-Liu-Moon-Shimizu in \cite{DLMS}, \cite{DLMSII} provides us the analytic prismatic $F$-crystals (as discussed in Subsec. \ref{subsec-1.2}). This can be used to recover Breuil-Kisin $\Gal_K$-modules as in Def. 1.1.8 in \cite{Gao} by evaluating at Breuil-Kisin prism and $A_{\text{inf}}$-prism. This is in the spirit of Rem. 1.1.10 (3) in \cite{Gao} on the comparison of Breuil-Kisin cohomology and $A_{\text{inf}}$-cohomology. In the classical case, the integral Kisin descent data, studied by Du-Liu in \cite{DL}, is the correct replacement of $(\varphi, \widehat{G})$-modules as $\fS^{(1)}$ is $p$-adically complete (described as $A_{\text{st}}^{(2)}$ in \cite{DL}). Moreover, the integral Kisin data is also available for relative setting e.g. imperfect residue field case, small affine case etc. by \cite{DLMS}, \cite{DLMSII}. In particular, the objects in the category $\textbf{DD}_{\cO_{\cE}}$ whose underlying étale $\varphi$-modules arises from finite height Breuil-Kisin modules can be thought of as relative version of weak $(\varphi, \widehat{G})$-modules and weak Breuil-Kisin $\Gal_K$-modules.\smallskip

\noindent If $L$ is a complete discrete valuation field (CDVF) of mixed characteristic $(0,p)$ with imperfect residue field having finite $p$-basis, then one can consider a corresponding CDVF $L_g$ which is a non-algebraic extension of $L$ and with perfect residue field (see Sec. \ref{sec-2} for more details). One can choose an embedding $\overline{L} \hookrightarrow \overline{L_g}$ of fixed respective algebraic closures. This induces a continuous group homomorphism $\Gal_{L_g} \rightarrow \Gal_L$. One can use Thm. 1.1 of \cite{Mor} to prove that if $T$ is a finite $E$-height $\Z_p$-representation of $\Gal_L$ then the $p$-adic representation $T[1/p]$ of $\Gal_L$ is potentially semistable (see Prop. \ref{mainthm-prelim}). However, this result does not detect the precise finite extension of $L$ over what $T[1/p]$ would be semistable. The novelty of the current paper is the use of integral Kisin descent data which gives us the semistability over the finite extension $L_m = L(\pi_m)$ which is reminiscent of Thm. \ref{Gao-main-perf}. Moreover, such phenomenon happens for the small affine case which is given by Thm. \ref{main-thm-1}.

\subsection{Methodology:}

Let us discuss the roadmap to prove Thm. \ref{main-thm-1}. Let $L$ be a CDVF of mixed characteristic $(0,p)$ containing $K$ and with imperfect residue field having finite $p$-basis. One can define the notion of finite $E$-height for a $\Z_p$-representation of $\Gal_L=\Gal(\overline{L}/L)$ in terms of étale $\varphi$-module that arises from a finite free Breuil-Kisin module as in Def. \ref{finiteEheightlocal system-def} (see Def. \ref{def-1}). The formalism of semistable Galois representation in imperfect residue field case can be attributed to \cite{Fon94}, \cite{Mor}, \cite{Ohk13}. By convention, we call a $\Z_p$-representation $T$ of $\Gal_L$ to be semistable if the corresponding $p$-adic representation $T[1/p]$ of $\Gal_L$ is semistable. The key intermediate step is the following analogue of Thm. \ref{Gao-main-perf}, which we prove in Sec. \ref{Sec-4}.

\begin{thm}\label{main-thm-2}
Assume that $m=\text{max} \{i : \zeta_{p^i} \in \widehat{K^{\text{un}}} \}$. Let $T$ be a finite $E$-height $\Z_p$-representation of $\Gal_L$. Then, $T|_{\Gal_{L_m}}$ is semistable $\Z_p$-representation of $\Gal_{L_m}$; where $L_m = L(\pi_m)$. 
\end{thm} 

    \noindent We use functoriality to describe the dictionary between the restriction of Galois representation and the pullback of descent data in Sec. \ref{sec-3}. The key observation is Prop. \ref{prop-imp-perf-st} which states that if $T$ is a finite $E$-height $\Z_p$-representation of $\Gal_L$ such that $T|_{\Gal_{L_g}}$ is semistable then $T$ is semistable $\Z_p$-representation of $\Gal_L$. The functoriality of the equivalence between Laurent $F$-crystals and integral Galois representations implies that the restriction of $T$ along $\Gal_{L_g} \rightarrow \Gal_L$ is equivalent to pullback of descent data along a map $\iota_g^{(1)}$ among respective Breuil-Kisin log prisms and their co-products (see Prop. \ref{functoriality-imperf-perf}). The proof of Prop. \ref{prop-imp-perf-st} relies on two technical lemmas : Lem. \ref{intersection-equality} and Lem. \ref{kernel-lemma} about such map $\iota_g^{(1)}$. For instance, Lem. \ref{intersection-equality} is analogous to Prop. 3.6 in \cite{Moon}. Then Thm. \ref{main-thm-2} can be deduced as follows. First we obtain semistability over the finite extension $L_g(\pi_m)$ of $L_g$ by Thm. \ref{Gao-main-perf}. So, $T|_{\Gal_{L_m}}$ is a finite $E(u)$-height $\Z_p$-representation of $\Gal_{L_m}$ whose restriction $T|_{\Gal_{L_g(\pi_m)}}$ is semistable. Hence, we use Prop. \ref{prop-imp-perf-st} to conclude that $T|_{\Gal_{L_m}}$ is semistable as per Def. \ref{etale-realization}. \smallskip

\noindent Finally, we use a purity result of \cite{DLMSII} to prove Thm. \ref{main-thm-1}. To describe the purity result we make a few more notations. Let us write the generic points of $X$ by $\{ \xi_1, \xi_2, \cdots, \xi_r\}$. For example, if $X=X^0$ (i.e. $R=R^0$) then $\xi_j=(X_j)$ for all $1 \leq j \leq r$. For each such point $\xi_j$ the completed local ring $\cO_{L_j}:=\cO_{X, \xi_j}^\wedge$ is a complete discrete valuation ring (CDVR) with uniformizer $\pi$. Let us write $\Delta_j = \mathrm{Spf}(\cO_{X, \xi_j}^\wedge)$. Write $L_j = \text{Frac}(\cO_{X, \xi_j}^\wedge)$. Then, for each $1 \leq j \leq r$, $L_j$ is a CDVF of mixed characteristic $(0,p)$ with imperfect residue field having finite $p$ basis. We have the morphisms $f_j : \Delta_j \rightarrow X$ of $p$-adic formal schemes for each $j = 1,2, \cdots, r$.  Therefore, we have the rank one points $f_j : \mathrm{Spa}(L_j, \cO_{L_j}) \rightarrow \cX$ in $\cX$. 

\begin{thm}(Thm. 1.6 of \cite{DLMSII})\label{purity-thm}
A $\Z_p$-local system $\bL$ over $\cX$ is semistable if and only if its pullback $f_j^\star \bL$ along the map $f_j$, regarded as $\Z_p$-representation of $\Gal_{L_j}$, is semistable for each $j =1,2,\cdots, r$.
 \end{thm} 
 
 \noindent Now, if a $\Z_p$-local system $\bL$ over $\cX$ is of finite $E$-height then for each $j =1,2,\cdots, r$ the pullback $f_j^\star \bL$ is finite $E$-height $\Z_p$-representation of $\Gal_{L_j}$. So, we can apply Thm. \ref{main-thm-2} to deduce that the restrictions $f_j^\star \bL|_{\Gal_{L_j(\pi_m)}}$ are semistable $\Z_p$-representations of $\Gal_{L_j(\pi_m)}$ for all $1 \leq j \leq r$. In particular, the further restrictions $f_j^\star \bL|_{\Gal_{L_j'}}$ are semistable $\Z_p$-representations of $\Gal_{L_j'}$ for all $1 \leq j \leq r$; where $L_j'=L_j(\pi_m, X_1^{1/p^m}, X_2^{1/p^m}, \dots, X_{j-1}^{1/p^m}, X_{j+1}^{1/p^m}, \dots, X_r^{1/p^m})$. Note that $L_j'$ is a finite extension of $L_j$ for each $1 \leq j \leq r$. The rest of the argument to prove Thm. \ref{main-thm-1} relies on the fact that the stalk of $\iota_m^\star \bL$ at the rank one point $\mathrm{Spa}(L_j', \cO_{L_j'}) \rightarrow \cX_m$ as a Galois representation of $\Gal_{L_{j}'}$ is identified with $f_j^\star \bL|_{\Gal_{L_j'}}$ by functoriality.
 
\subsection{Organization:}
We now describe the organization of the rest of the paper. In Sec. \ref{sec-2}, we set up the basic notations and recall the elementary facts about absolute log prismatic site (see Subsec. \ref{sec-2-2} and \ref{sec-2-3}). We recall the category of descent data and its equivalence with the category of integral Galois representations in Subsec. \ref{cat-eq}. Moreover, we describe the same for semistable integral representations and finite $E$-height representations in Subsec. \ref{sec-2-5} and Subsec. \ref{sec-3}, respectively. In Sec. \ref{Sec-4}, we prove the required lemmas on maps of (coproduct of) Breuil-Kisin log prisms (see Lem. \ref{intersection-equality} and Lem. \ref{kernel-lemma}) and prove Thm. \ref{main-thm-2} (see Thm. \ref{prop-imp-perf-st}). In Sec. \ref{sec-5}, we consider the small affine case. We first set up the small affine case in Subsec. \ref{sec-5-1} and discuss the equivalence between pullback of descent data and the pullback of $\Z_p$-local system in Subsec. \ref{sec-5-2}. Finally, we complete the proof of Thm. \ref{main-thm-1} in Subsec. \ref{sec-5-3}.
 
 \vspace{15pt}
 
\noindent \textbf{Acknowledgement:} The author is grateful to Abhinandan for carefully answering questions about Prop. \ref{prop-imp-perf-st} and for several useful comments. The author sincerely thanks Hui Gao for his invaluable suggestions. He also thanks Aditya Karnataki, Aprameyo Pal, Arnab Saha, and Jishnu Roy for many fruitful discussions. Part of this work was carried out during the author's visit to HRI. The author thanks HRI for its warm hospitality. This work was supported by the Prime Minister's Research Fellowship (PMRF), Govt. of India. 

\section{Preliminaries}\label{sec-2}

\subsection{Basic Set up and Notation:}\label{sec-2-1}
Let $p$ be a fixed odd prime. Let $K$ be a finite extension of $\Q_p$ with ring of integers $\cO_K$ and $k$ be its residue field. Choose a uniformizer $\pi$ in $\cO_K$. The maximal unramified extension of $\Q_p$ inside $K$ is $K_0:=W(k)[1/p]$. Let $e=[K:K_0]$ be the ramification index. We write $E(u) \in W(k)[u]$ for the minimal polynomial of $\pi$. Fix a $p$-power compatible sequence $\overrightarrow{\pi} =\{ \pi, \pi^{1/p}, \pi^{1/p^2}, \cdots \}$ of $\pi$ inside a fixed algebraic closure $\overline{K}$. For this choice, we have the arithmetic Kummer tower $K_{\infty}=\cup_{n} K_n$; where, $K_n=K(\pi^{1/p^n})$ for all $n$. We write $\pi_n := \pi^{1/p^n}$. Clearly, $K_{\infty}$ is a totally ramified tower and the completion of $K_{\infty}$ is a perfectoid field.\smallskip

\noindent Consider the imperfect field $k(\overline{X_1}, \overline{X_2}, \cdots, \overline{X_b})$ in characteristic $p$ having finite $p$-basis. Let $\cO_{L_0}$ be a Cohen ring with residue field $k(\overline{X_1}, \overline{X_2}, \cdots, \overline{X_b})$ and $L_0=\cO_{L_0}[1/p]$. We fix a choice of lift $X_i$ for each $\overline{X_i}$ in $\cO_{L_0}$ and equip $\cO_{L_0}$ with a Frobenius $\varphi$ such that $\varphi(X_i)=X_i^p$ for $1 \leq i \leq b$. We write $\cO_{L}=\cO_K \otimes_{W(k)} \cO_{L_0}$. So, $\pi$ is a uniformizer in $\cO_L$. Its fraction field $L:=\cO_{L}[1/\pi]$ is a complete discretely valued field of mixed characteristic $(0,p)$ with imperfect residue field $k(\overline{X_1}, \overline{X_2}, \cdots, \overline{X_b})$.  We furthermore fix a $p$-power compatible sequence $\overrightarrow{X_i} = \{ X_{i}, X_i^{1/p}, X_i^{1/p^2}, \cdots \}$ for every $X_i$ inside a fixed algebraic closure $\overline{L}$. With respect to this choice, we have  the tower $$L_{\infty} = \cup_{n} L(\pi^{1/p^n}, X_1^{1/p^n}, X_2^{1/p^n}, \cdots, X_b^{1/p^n}).$$ 
Note that, this tower is \textit{not} totally ramified as the corresponding residue fields are not same. There is a totally ramified subtower $L_{\infty}^{\text{imp}}=\cup_n L(\pi_n)$ inside $L_{\infty}$.

\begin{exa}\label{example2}
Consider $\tilde{R^0}:=W(k) \langle X_1, X_2, \cdots, X_r, X_{r+1}^{\pm}, \cdots X_{b}^{\pm} \rangle / (X_1. X_2 \dots X_r - p)$ and $R^0=\cO_K \otimes_{W(k)} \tilde{R^0}$. For each $ 1 \leq j \leq r$, let $(\tilde{R^0})^\wedge_{(X_j)}$ denote the $p$-adic completion of the localization of $\tilde{R^0}$ at $(X_j)$. Then, $(\tilde{R^0})^\wedge_{(X_j)}$ is a Cohen ring with residue field $k(\overline{X_1}, \overline{X_2}, \cdots, \overline{X_{j-1}}, \overline{X_{j+1}}, \cdots, \overline{X_b})$. So, the $p$-adic completion $(R^0)^\wedge_{(X_j)}$ of the localization of $R^0$ at $(X_j)$ (resp. $(R^0)^\wedge_{(X_j)}[1/p]$) is a complete discretely valued ring (resp. field) of mixed characteristic $(0,p)$ with uniformizer $\pi$ and with imperfect residue field $k(\overline{X_1}, \overline{X_2}, \cdots, \overline{X_{j-1}}, \overline{X_{j+1}}, \cdots, \overline{X_b})$. We write $\cO_{L_{0,j}}=(\tilde{R^0})^\wedge_{(X_j)}$ and $\cO_{L_j}=(R^0)^\wedge_{(X_j)}$ for each $1 \leq j \leq r$.
\end{exa}

\noindent We are interested in the $\Z_p$-lattices in finite dimensional continuous $p$-adic representations of the absolute Galois group $\Gal_L:=\Gal(\overline{L} / L)$ and denote the category by $\text{Rep}_{\Z_p}(\Gal_L)$. Similarly, write the category $\text{Rep}_{\Z_p}(\Gal_{L_{\infty}})$ for $\Z_p$-representations of $\Gal_{L_{\infty}}=\Gal(\overline{L}/ L_{\infty})$.  Let $\fS_L=\cO_{L_0} [[u]]$ and $\cO_{\cE,L} = \fS_L[E(u)^{-1}]^\wedge_p$. Both of $\fS_L$ and $\cO_{\cE,L}$ are equipped with the Frobenius endomorphisms $\varphi_{\fS_L}$ and $\varphi_{\cO_{\cE,L}}$; defined by $u \mapsto u^p$ and $X_i \mapsto X_i^p$. Let $\overline{\mathcal{O}_L}$ be the union of finite $\mathcal{O}_L$-subalgebras of the fixed algebraic closure $\overline{L}$ of $L$. Let $\overline{\mathcal{O}_L}^{\flat}$ be the tilt of the $p$-adic completion of $\overline{\mathcal{O}_L}$. We have a $\varphi$-equivariant embedding $\mathfrak{S}_L \rightarrow W(\overline{\mathcal{O}_L}^{\flat}[(\pi^{\flat})^{-1}])$ by $u \mapsto [\pi^{\flat}]$ and $X_i \mapsto [X_i^\flat]$. This extends to an embedding $\mathcal{O}_{\mathcal{E},L} \rightarrow W(\overline{\mathcal{O}_L}^{\flat}[(\pi^{\flat})^{-1}])$. Let $\mathcal{O}_{\mathcal{E},L}^{ur}$ be the union of finite étale $\mathcal{O}_{\mathcal{E},L}$-subalgebras of $W(\overline{\mathcal{O}_L}^{\flat}[(\pi^{\flat})^{-1}])$. Denote by $\widehat{\mathcal{O}_{\mathcal{E},L}^{ur}}$ its $p$-adic completion.

\begin{defn}
A finite height Breuil-Kisin module $\mathfrak{M}_L$ over $\mathfrak{S}_L$ is a finite free $\mathfrak{S}_L$-module equipped with a Frobenius semilinear endomorphism $\varphi_{\mathfrak{M}_L}$ such that the cokernel of its Frobenius linearization $\varphi^*_{\mathfrak{M}_L}=1 \otimes \varphi_{\mathfrak{M}_L}$ on $\mathfrak{S}_L \otimes_{\varphi, \mathfrak{S}_L} \mathfrak{M}_L$ is killed by some power of $E(u)$. In other words, there is a non-negative integer $r$ such that $(E(u))^r (\mathfrak{M}_L) \subset \im (\varphi_{\mathfrak{M}_L}^*)$. The category of finite height Breuil-Kisin modules is denoted by $\text{Mod}^{\varphi}_{\fS_L}$.  
\end{defn}

\begin{defn}\label{def-1}
A lattice $T\in \text{Rep}_{\Z_p}(\Gal_L)$ is said to be of finite $E(u)$-height (w.r.t. $\overrightarrow{\pi}$ and $\overrightarrow{X_i}$) if there exists a finite height Breuil-Kisin module $\mathfrak{M}_L$ over $\mathfrak{S}_L$ such that $$(\mathfrak{M}_L \otimes_{\mathfrak{S}_L} \widehat{\mathcal{O}_{\mathcal{E},L}^{ur}})^{\varphi=1} = T|_{\Gal_{L_{\infty}}}.$$
\end{defn}

\begin{rem}
This is analogous to Def. 1.1.15 of \cite{Gao} in perfect residue field case. For our purpose, we want to define finite $E(u)$-height $\Z_p$-lattice in an alternative but equivalent way (see Def. \ref{def-2} below).
\end{rem} 

\noindent Let us denote the category of free $\Z_p$-lattices in $p$-adic representations of $\Gal_{L_{\infty}}$ by $\text{Rep}_{\Z_p}^{pr}(\Gal_{L_{\infty}})$. 

\begin{defn}
An étale $(\varphi, \mathcal{O}_{\mathcal{E},L})$-module $(\cM_L, \varphi_{\cM_L})$ consists of a finitely generated module $\mathcal{M}_L$ over $\mathcal{O}_{\mathcal{E},L}$ equipped with a $\varphi$-semilinear endomorphism $\varphi_{\mathcal{M}_L}$ whose $\varphi$-linearization $1 \otimes \varphi_{\mathcal{M}_L}$ on $\mathcal{O}_{\mathcal{E},L} \otimes_{\varphi, \mathcal{O}_{\mathcal{E},L}} \mathcal{M}_L$ is an isomorphism. The category of étale $(\varphi, \cO_{\cE,L})$-modules is denoted by $\text{Mod}^{\varphi}_{\cO_{\cE,L}}$. It is called projective if the underlying $\mathcal{O}_{\mathcal{E},L}$-module $\mathcal{M}_L$ is projective. Denote by $\text{Mod}^{pr,\varphi}_{\mathcal{O}_{\mathcal{E},L}}$ the category of projective étale $(\varphi, \mathcal{O}_{\mathcal{E},L})$-module.
\end{defn} 

\noindent We have the following equivalence due to \cite{Kim}. 
\begin{thm}\label{Kim}(Prop. 7.7 of \cite{Kim} and Prop. 2.5 of \cite{LM20})\label{Kim-eq}
The exact tensor functors $\mathcal{T}$ from $\text{Mod}^{pr,\varphi}_{\mathcal{O}_{\mathcal{E},L}}$ to $\text{Rep}_{\Z_p}^{pr}(\Gal_{L_{\infty}})$ defined by $\mathcal{T}(\mathcal{M}_L):= (\widehat{\mathcal{O}_{\mathcal{E},L}^{ur}} \otimes_{\mathcal{O}_{\mathcal{E},L}} \mathcal{M}_L)^{\varphi=1}$ and $\mathcal{M}$ from $\text{Rep}_{\Z_p}^{pr}(\Gal_{L_{\infty}})$ to $\text{Mod}^{pr,\varphi}_{\mathcal{O}_{\mathcal{E},L}}$ defined by $\mathcal{M}(T):=(\widehat{\mathcal{O}_{\mathcal{E},L}^{ur}} \otimes_{\Z_p} T)^{\Gal_{L_{\infty}}}$ are quasi-inverse of each other.
\end{thm}  


\noindent  Now we define a finite $E(u)$-height $\Z_p$-lattice in $\text{Rep}_{\Z_p}(\Gal_{L_{\infty}})$ as follows.

\begin{defn}\label{def-2}
An object $T$ in $\text{Rep}_{\Z_p}(\Gal_{L_{\infty}})$ is said to be of finite $E(u)$-height if there exists an object $(\fM_L, \varphi_{\fM_L})$ in $\text{Mod}_{\fS_L}^{\varphi}$ such that the associated étale  $(\varphi, \mathcal{O}_{\mathcal{E},L})$-module $\cM(T):=\mathcal{M}_L$ can be written as $\cM_L= \mathfrak{M}_L \otimes_{\mathfrak{S}_L} \mathcal{O}_{\mathcal{E},L}$. A $\Z_p$-representation $T$ of $\Gal_L$ is called of finite $E(u)$-height if its restriction $T|_{\Gal_{L_{\infty}}}$ is so.
\end{defn}

\noindent Note that Def. \ref{def-2} depends on the fixed choice of $\overrightarrow{\pi}$ and $\overrightarrow{X_i}$. Indeed, these choices determine the inclusion $\fS_L \rightarrow \widehat{\cO_{\cE,L}^{ur}}$ defined by $u \mapsto [\pi^\flat], X_i \mapsto [X_i^\flat]$. It is easy to see that the Def. \ref{def-1} and Def. \ref{def-2} are equivalent. \smallskip

\noindent If a $\Z_p$-lattice $T$ in   $\text{Rep}_{\Z_p}(\Gal_{L_{\infty}})$ is of finite $E(u)$-height in the sense of the definition above, then $\mathfrak{M}_L \otimes_{\mathfrak{S}_L} \mathcal{O}_{\mathcal{E},L}$ is a finite free étale $(\varphi, \mathcal{O}_{\mathcal{E},L})$-module and $T=\mathcal{T}(\mathfrak{M}_L \otimes_{\mathfrak{S}_L} \mathcal{O}_{\mathcal{E},L})$. So, by Thm. \ref{Kim}, $T$ is finite free $\Z_p$-lattice in $\text{Rep}_{\Z_p}(\Gal_{L_{\infty}})$. Finally, a $p$-adic representation $V$ of $\Gal_L$ is said to be of finite $E(u)$-height if $V$ contains a $\Gal_L$-stable $\Z_p$-lattice $T$ which is of finite $E(u)$-height. We are now ready to mention our primary result. 

\noindent

\begin{prop}\label{mainthm-prelim}
Let $L$ be the complete discretely valued field of mixed characteristic $(0,p)$ with residue field $k(\overline{X_1}, \overline{X_2}, \cdots, \overline{X_b})$ as above. Then, a finite $E(u)$-height $p$-adic representation of $\Gal_L$ is potentially semistable. 
\end{prop}

\begin{proof}
Consider a finite $E(u)$-height $p$-adic representation $V$ of $\Gal_L$. Let $k_g$ be the coperfection of $k(\overline{X_1}, \overline{X_2}, \cdots, \overline{X_b})$ i.e. $$k_g := \lim_{\varphi} k(\overline{X_1}, \overline{X_2}, \cdots, \overline{X_b})= k(\overline{X_1}^{1/p^{\infty}}, \overline{X_2}^{1/p^{\infty}}, \cdots, \overline{X_b}^{1/p^{\infty}}).$$ Denote $\mathcal{O}_{L_0,g}=W(k_g)$ and $\mathcal{O}_{L_g}:=\mathcal{O}_{L_0,g} \otimes_{W(k)} \mathcal{O}_K$. Then $L_g =\mathcal{O}_{L_g}[1/\pi]$ is a complete discretely valued field with perfect residue field $k_g$. The natural map $\mathcal{O}_{L_0} \rightarrow \mathcal{O}_{L_{0, g}} : X_i \mapsto [X_i^\flat]$ induces $\cO_{L} \rightarrow \cO_{L_g}$. In particular, we have the continuous map $\Gal_{L_g} \rightarrow \Gal_L$ of the corresponding absolute Galois groups.\smallskip

\noindent Now let $\mathfrak{S}_{L_g}=\mathcal{O}_{L_{0, g}}[[u]]$. The natural map $\mathcal{O}_{L_0} \rightarrow \mathcal{O}_{L_{0, g}} : X_i \mapsto [X_i^\flat]$ induces a ring map $i_g : \mathfrak{S}_L \rightarrow \mathfrak{S}_{L_g}$ which moreover provides the ring extension $\cO_{\cE,L} \rightarrow \cO_{\cE, L_g}$. Let us denote the category of étale $\varphi$-module over $\cO_{\cE,L_g}$ by $Mod^{{\varphi}}_{\cO_{\cE,L_g}}$. Then, by Prop. 4.2.5 in \cite{Gao-imp}, the scalar extension functor from $Mod^{\varphi}_{\cO_{\cE,L}}$ to $Mod^{\varphi}_{\cO_{\cE,L_g}}$ is an equivalence.\smallskip

\noindent Let $T$ be a finite height $\Z_p$-lattice inside the given $p$-adic representation $V$ of $\Gal_L$. By Thm. \ref{Kim}, there is a finite free étale $\varphi$-module $\cM_L$ over $\cO_{\cE,L}$ associated to $T|_{\Gal_{L_{\infty}}}$. Note that the image of $L_{\infty} \hookrightarrow L_{g, \infty}$ is dense i.e. $\Gal_{L_{\infty}} \cong \Gal_{L_{g, \infty}}$; where $L_{g,\infty}=\cup_{n\geq 1} L_g(\pi_n)$. Therefore, if the étale $\varphi$-module associated to $T|_{\Gal_{L_{g, \infty}}}$ is given by $(\cM_{L_g}, \varphi_{\cM_{L_g}})$, then $\cM_{L_g}=\cO_{\cE,L_g} \otimes_{\cO_{\cE,L}} \cM_L$. Since, $T$ is of finite $E(u)$-height, there exists a finite free $\mathfrak{S}_L$-module $\mathfrak{M}_L$ such that $\cM_L = \fM_L \otimes_{\fS_L} \cO_{\cE,L}$. Consider the $\mathfrak{S}_{L_g}$-module $\mathfrak{M}_{L_g}:=\mathfrak{M}_L \otimes_{\mathfrak{S}_L} \mathfrak{S}_{L_g}$. It is a finite free $\fS_{L_g}$-module. We have,
$$\mathcal{M}_{L_g} = \mathcal{M}_L \otimes_{\mathcal{O}_{\mathcal{E},L}} \mathcal{O}_{\mathcal{E},L_g}
=(\mathfrak{M}_L \otimes_{\mathfrak{S}_L} \mathcal{O}_{\mathcal{E},L_g})
=(\mathfrak{M}_L \otimes_{\mathfrak{S}_L} \mathfrak{S}_{L_g}) \otimes_{\mathfrak{S}_{L_g}} \mathcal{O}_{\mathcal{E},L_g}
=\mathfrak{M}_{L_g} \otimes_{\mathfrak{S}_{L_g}} \mathcal{O}_{\mathcal{E},L_g}.$$

\noindent Therefore, $T|_{\Gal_{L_g}}$ is of finite height $\Z_p$-representation of $\Gal_{L_g}$. By Thm. \ref{Gao-main-perf}, $V=T[1/p]|_{\Gal_{L_g}}$ is potentially semistable. By Thm. 1.1 of \cite{Mor}, $V$ is a potentially semistable representation of $\Gal_L$. 
\end{proof}

\begin{rem}
As mentioned in Subsec. \ref{sec-1-3} the precise finite extension of $L$ over that the $p$-adic representation $V=T[1/p]$ is semistable cannot be known from the above proposition. In the spirit of  Thm. \ref{Gao-main-perf} one can expect to get the desired finite extension along the Kummer (sub-)tower $L_{\infty}^{\text{imp}} = \cup_n L(\pi_n)$. This is given by our main result Thm. \ref{mainthm-cdvf} in Sec. \ref{Sec-4}. 
\end{rem}

\vspace{15pt}

\subsection{Absolute Log Prismatic Site over $\cO_{L}$}\label{sec-2-2}
\noindent In this subsection, we define the Breuil-Kisin log prism in the absolute log prismatic site $(Y, M_Y)_{\Prism}$ over the formal scheme $Y=\mathrm{Spf}(\cO_L)$ equipped with certain log structure $M_Y$ and describe its co-products. For details, see Subsec. 2.3 of \cite{DLMSII}. See also Sec. 5 of \cite{DL} for the same over $\mathrm{Spf}(\cO_{K})$. \smallskip

 \noindent We consider the log formal scheme $(Y, M_Y)$, where $Y=\mathrm{Spf} (\cO_{L})$ and the log structure $M_Y$ is given by the pre-log structure $\N \rightarrow \cO_L$ : $n \mapsto \pi^n$. We denote the corresponding absolute log prismatic site by $(Y , M_Y)_{\Prism}$ as per \cite{KosI}. The object $(\fS_L, E(u), M_{\mathrm{Spf} (\fS_L)})$ in $(Y , M_Y)_{\Prism}$ is called the Breuil-Kisin log prism; where, $\fS_L= \cO_{L_0}[[u]]$, equipped with Frobenius $\varphi(u)=u^p$, $\varphi(X_i)=X_i^p$, and the log structure $M_{\mathrm{Spf} (\fS_L)}$, given by $\N \rightarrow \fS_L$: $n \mapsto u^n$. In practice, we write the Breuil-Kisin log prism as $(\fS_L, E(u), \N)^a$. This object covers the final object of the topos associated to $(Y , M_Y)_{\Prism}$ (see Lem. 2.8 in \cite{DLMSII}). \smallskip
 
 \noindent Let $(\fS_L^{(i)}, E(u), \N^{i+1})^a$ denote the $(i+1)$-th self-coproduct of the Breuil-Kisin log prism in $(Y , M_Y)_{\Prism}^{\text{opp}}$. For instance, $\fS_L^{(1)}$ has the following description (see Subsec. 2.3 in \cite{DLMSII}): 
 
 $$ \fS_L^{(1)}:= \fS_L[[1-\frac{u_2}{u}, 1-\frac{X_{1,2}}{X_{1}}, 1-\frac{X_{2,2}}{X_{2}}, \cdots, 1-\frac{X_{b,2}}{X_{b}}]]\{ \dfrac{1-\frac{u_2}{u}}{E(u)}, \dfrac{1-\frac{X_{1,2}}{X_{1}}}{E(u)}, \dfrac{1-\frac{X_{2,2}}{X_{2}}}{E(u)} ,\cdots, \dfrac{1-\frac{X_{b,2}}{X_{b}}}{E(u)}  \}^\wedge_{\delta}.$$ Here, $\{.\}^\wedge_{\delta}$ denotes the derived $(p, E)$-completion in the category of $\delta$-algebras. It comes with two projection maps $p_1, p_2 : \fS_L \rightarrow \fS_L^{(1)}$, which are classically faithfully flat by  Lem. 2.17 of \cite{DLMSII}. We have, $p_1(X_i)=X_{i}$, $p_2(X_i)=X_{i,2}$ for $1 \leq i \leq b$ and $p_2(u)=u_2$. In particular, the prism $(\fS_L^{(1)}, E(u))$ is transversal i.e. $\fS_L^{(1)}/ E(u)$ is $p$-torsion free. It is naturally equipped with log structure $\N^2 \rightarrow \fS_L^{(1)}: (s,0) \mapsto u^s; (0,t) \mapsto u_2^t$. We write the resulting log prism as $(\fS_L^{(1)}, E(u), \N^2)^a$.\smallskip
 
\noindent Similarly, for each $L_n=L(\pi_n)$, we have $(\fS_{L_n}, E(u^{p^n}), \N)^a$ and its self-coproduct $(\fS_{L_n}^{(1)}, E(u^{p^n}), \N)^a$ in the absolute log prismatic (opp-)site $(\mathrm{Spf}(\cO_{L_n}), M_{\mathrm{Spf}(\cO_{L_n})})_{\Prism}^{\text{opp}}$. Beware that since $L_n$ is totally ramified over $L$ we have $\fS_{L_n}=\fS_{L}=\cO_{L_0}[[u]]$. We can regard $(\fS_{L_n}, E(u^{p^n}), \N)^a$ as an object of $(Y, M_Y)_{\Prism}$ via $\cO_{L} \hookrightarrow \cO_{L_n} \cong \fS_{L_n} / E(u^{p^n})$. It admits a morphism $\iota_n : (\fS_{L}, E(u), \N)^a \rightarrow (\fS_{L_n}, E(u^{p^n}), \N)^a$ given by $u \mapsto u^{p^n}$ in $(Y, M_Y)_{\Prism}$ by the following lemma. 


\begin{lem}\label{Frobenius}
The map $\iota_n : \fS_L \rightarrow \fS_{L_n}$ defined by $\sum a_i u^i \mapsto \sum a_i u^{ip^n}$ induces a map of log prisms $\iota_n : (\fS_L, E, \N)^a \rightarrow (\fS_{L_n}, E_n, \N)^a$; where $E:=E(u)$, $E_n:=E(u^{p^n})$ and the map of constant monoids $\N \rightarrow \N$ is given by $1 \mapsto p^n$.
\end{lem}

\begin{proof}
It is straightforward to check that $\iota_n$ is a map of prisms. Note that the following diagram commutes:
\begin{equation}
\begin{tikzcd}
\mathbb{N} \arrow[r, "1 \mapsto p^n"] \arrow[d, "1 \mapsto u"'] & \mathbb{N} \arrow[d, "p^n \mapsto u^{p^n}"] \\
\fS_L \arrow[r, " u \mapsto u^{p^n}"]                        & \fS_{L_n}
\end{tikzcd}
\end{equation}
This yields $\iota_n$ to be a map of log prisms.
\end{proof}

\noindent Now, the map $\iota_n : (\fS_L, E, \N)^a \rightarrow (\fS_{L_n}, E_n, \N)^a$ induces the following pushout diagram in the absolute log prismatic (opp-)site $(Y, M_Y)^{\text{opp}}_{\Prism}$, which gives rise to the morphism $\iota_n^{(1)} : (\fS_L^{(1)}, E, \N^2)^a \rightarrow (\fS_{L_n}^{(1)}, E_n, \N^2)^a$.  
\begin{equation}\label{pushout-sec-2-2}
\begin{tikzcd}
(\fS_{L_n}, E_n, \N)^a \arrow[rr, "p_1"]                 &                                 & (\fS_{L_n}^{(1)}, E_n, \N^2)^a                  \\
(\fS_L, E, \N)^a \arrow[u, "\iota_n"] \arrow[r, "p_1"']  & (\fS_L^{(1)}, E, \N^2)^a \arrow[ru, "\iota_n^{(1)}"']              &                    \\
                                  & (\fS_L, E, \N)^a \arrow[u, "p_2"] \arrow[r, "\iota_n"']& (\fS_{L_n}, E_n, \N)^a \arrow[uu, "p_2"']
\end{tikzcd}
\end{equation}

 \vspace{15pt} 
 
\subsection{Breuil-Kisin log prism over $\cO_{L_g}$:}\label{sec-2-3}
\noindent Similar to Subsec. \ref{sec-2-2}, we consider the Breuil-Kisin prism and its co-products over the complete discrete valuation ring $\cO_{L_g}$ with perfect residue field from the proof of Prop. \ref{mainthm-prelim}. Recall that $\cO_{L_{0,g}}=W(k_g)$; where $k_g = k(\overline{X_1}^{ 1/p^{\infty}}, \overline{X_2}^{ 1/p^{\infty}}, \cdots, \overline{X_b}^{ 1/p^{\infty}})$ and $\cO_{L_g} := \cO_{L_{0,g}} \otimes_{W(k)} \cO_K$. Moreover, we have the map $\iota_g : \cO_{L_0} \rightarrow \cO_{L_{0,g}} : X_i \mapsto [X_i^\flat]$ which moreover induces $\cO_L \rightarrow \cO_{L_g}$. Let $\fS_{L_g} := \cO_{L_{0,g}}[[u]]$. So, we have the map $\iota_g : \fS_L \rightarrow \fS_{L_g}$ of $\delta$-rings. \smallskip

\noindent We have the Breuil-Kisin log prism $(\fS_{L_g}, E, \N)^a$ with the log structure given by $\N \rightarrow \fS_{L_g} : n \mapsto u^n$ in the absolute log prismatic site $(\mathrm{Spf}(\cO_{L_g}) , M_{\mathrm{Spf}(\cO_{L_g})})_{\Prism}$. Similar to Subsec. \ref{sec-2-2}, $(\fS_{L_g}, E, \N)^a$ can be regarded as an object of $(Y, M_Y)_{\Prism}$ via $\cO_L \rightarrow \cO_{L_g} \cong \fS_{L_g} / E(u)$ and it admits the map $\iota_g : (\fS_L, E, \N)^a \rightarrow (\fS_{L_g}, E,  \N)^a$ in the site $(Y, M_Y)_{\Prism}^{\text{opp}}$. It is a covering map in $(Y, M_Y)_{\Prism}^{\text{opp}}$ since the map of underlying $\delta$-rings $\iota_g : \fS_L \rightarrow \fS_{L_g}$ is classically faithfully flat. \smallskip

 \noindent Let $(\fS^{(1)}_{L_g}, E, \N^2)^a$ be the self-coproduct of $(\fS_{L_g}, E, \N)^a$ in the site $(Y, M_Y)_{\Prism}^{\text{opp}}$ (and therefore in the site $(\mathrm{Spf}(\cO_{L_g}) , M_{\mathrm{Spf}(\cO_{L_g})})_{\Prism}^{\text{opp}}$ as well). The underlying $\delta$-ring $\fS^{(1)}_{L_g}$ is given by (cf. Subsec. 2.3 and proof of Thm. 5.0.12 in \cite{DL}) $$ \fS_{L_g}^{(1)} = \fS_{L_g}[[1-\frac{u_2}{u}]]\{ \dfrac{1-\frac{u_2}{u}}{E(u)}\}^\wedge_{\delta}.$$  The map $\iota_g : (\fS_L, E, \N)^a \rightarrow (\fS_{L_g}, E,  \N)^a$ gives rise to the following pushout diagram in $(Y, M_Y)_{\Prism}^{\text{opp}}$:

\begin{equation}\label{pushout-sec-2-3}
\begin{tikzcd}
(\fS_{L_g}, E, \N)^a \arrow[rr, "p_1"]                 &                                 & (\fS^{(1)}_{L_g}, E, \N^2)^a                  \\
(\fS_L, E, \N)^a \arrow[u, "\iota_g"] \arrow[r, "p_1"']  & (\fS_L^{(1)}, E, \N^2)^a \arrow[ru, "\iota_g^{(1)}"']              &                    \\
                                  & (\fS_L, E, \N)^a \arrow[u, "p_2"] \arrow[r, "\iota_g"']& (\fS_{L_g}, E, \N)^a \arrow[uu, "p_2"']
\end{tikzcd}
\end{equation}

\vspace{15pt}

\section{Finite height descent datum : CDVR case}

\subsection{Categorical Equivalence:}\label{cat-eq}

\noindent In this subsection, we describe the category of descent data associated to the category of $\Z_p$-lattices in $p$-adic Galois representations. We would treat the two cases in parallal: over $\cO_{L}$ i.e. the imperfect residue field case and over $\cO_{L_g}$ i.e. the perfect residue field case.\smallskip

\begin{defn}\label{defn-phitaumodule}
\noindent Consider the category $\textbf{DD}_{\cO_{\cE,L}}$ (resp. $\textbf{DD}_{\cO_{\cE,L_g}}$ ) of triplets $(\cM_L, \varphi_{\cM_L}, f_{\text{ét}})$ (resp. $(\cM_{L_g}, \varphi_{\cM_{L_g}}, f_{\text{ét},g})$), where
\begin{itemize}
\item $\cM_L$ (resp. $\cM_{L_g}$) is a finite free $\cO_{\cE,L}$-module (resp. $\cO_{\cE,L_g}$) equipped with a $\varphi_{\cO_{\cE,L}}$-semilinear (resp. $\varphi_{\cO_{\cE,L_g}}$) endomorphism $\varphi_{\cM_L}$ (resp. $\varphi_{\cM_{L_g}}$) such that the linearization $\mathrm{id} \otimes \varphi_{\cM_L} : \varphi^{\star} \cM_L \rightarrow \cM_L$ (resp. $\mathrm{id} \otimes \varphi_{\cM_{L_g}} : \varphi^{\star} \cM_{L_g} \rightarrow \cM_{L_g}$) is an isomorphism; i.e. $(\cM_L, \varphi_{\cM_L})$ (resp. $(\cM_{L_g}, \varphi_{\cM_{L_g}})$) is an étale $\varphi$-module over $\cO_{\cE,L}$ (resp. $\cO_{\cE,L_g}$),

\item $f_{\text{ét}}$ (resp. $f_{\text{ét},g}$) is an isomorphism $$f_{\text{ét}} : \fS_L^{(1)}[E^{-1}]^\wedge_p \otimes_{p_1, \cO_{\cE,L}} \cM_L \longrightarrow \fS^{(1)}_L[E^{-1}]^\wedge_p \otimes_{p_2, \cO_{\cE,L}} \cM_L$$ 
$$ (\text{resp.} \hspace{25pt}  f_{\text{ét},g} : \fS_{L_g}^{(1)}[E^{-1}]^\wedge_p \otimes_{p_1, \cO_{\cE,L_g}} \cM_{L_g} \longrightarrow \fS^{(1)}_{L_g}[E^{-1}]^\wedge_p \otimes_{p_2, \cO_{\cE,L_g}} \cM_{L_g}) $$
compatible with Frobenii and satisfies cocycle condition over $\fS^{(2)}_L[E^{-1}]^\wedge_p$ (resp. over $\fS^{(2)}_{L_g}[E^{-1}]^\wedge_p$). The morphisms are $\cO_{\cE,L}$-linear (resp. $\cO_{\cE,L_g}$-linear) morphism compatible with all structures.

\end{itemize}

\end{defn}

\begin{rem}
In the above definition, $\fS^{(i)}_L$ (resp. $\fS_{L_g}^{(i)}$) from Subsec. \ref{sec-2-2} (resp. \ref{sec-2-3}) is underlying $\delta$-ring of $(i+1)$-th self-coproduct of the prism $(\fS_L, E, \N)^a$ (resp. $(\fS_{L_g}, E, \N)^a$) for $i=1,2$ and $p_1, p_2 : \cO_{\cE,L}=\fS_L[E^{-1}]^\wedge_p \rightarrow \fS^{(1)}_L[E^{-1}]^\wedge_p$ (resp. $p_1, p_2 : \cO_{\cE,L_g}=\fS_{L_g}[E^{-1}]^\wedge_p \rightarrow \fS^{(1)}_{L_g}[E^{-1}]^\wedge_p$) are induced from the usual projection map $p_1, p_2 : \fS_L \rightarrow \fS^{(1)}_L[E^{-1}]^\wedge_p$ (resp. $p_1, p_2 : \fS_{L_g} \rightarrow \fS^{(1)}_{L_g}[E^{-1}]^\wedge_p$). Our category $\textbf{DD}_{\cO_{\cE,L_g}}$ is precisely the category of étale $\varphi$-module over $\fS_{L_g}[E^{-1}]^\wedge_p$ with descent data over $\fS_{L_g}^{(1)}[E^{-1}]^\wedge_p$ as in Def. 4.2.1 of \cite{DL}.
\end{rem}

Recall from Subsec. \ref{sec-2-2} and \ref{sec-2-3} that we have the absolute log prismatic sites $ (Y, M_Y)_{\Prism}=(\mathrm{Spf}(\cO_L), M_{\mathrm{Spf}(\cO_L)})_{\Prism}$ and $(Y_g, M_{Y_g})_{\Prism}= (\mathrm{Spf}(\cO_{L_g}), M_{\mathrm{Spf}(\cO_{L_g})})_{\Prism}$. We recall the definition of Laurent $F$-crystals over these sites. 

\begin{defn}
A Laurent $F$-crystal on $(Y, M_Y)_{\Prism}$ (resp. on $(Y_g, M_{Y_g})_{\Prism}$) consists of a pair $(\cE, \varphi_{\cE})$; where $\cE$ is a crystal in vector bundles on the ringed site $((Y, M_Y)_{\Prism}, \cO_{\Prism}[\cI_{\Prism}^{-1}]^\wedge_p)$ (resp. on the ringed site $((Y_g, M_{Y_g})_{\Prism}, \cO_{\Prism}[\cI_{\Prism}^{-1}]^\wedge_p)$) and $\varphi_{\cE} : \varphi^\star \cE \rightarrow \cE$ is an isomorphism. The category of Laurent $F$-crystals on $(Y, M_Y)_{\Prism}$ and $(Y_g, M_{Y_g})_{\Prism}$ are denoted by $\text{Vect}((Y, M_Y)_{\Prism}, \cO_{\Prism}[\cI_{\Prism}^{-1}]^\wedge_p)^{\varphi =1}$ and $\text{Vect}((Y_g, M_{Y_g})_{\Prism}, \cO_{\Prism}[\cI_{\Prism}^{-1}]^\wedge_p)^{\varphi =1}$) respectively. 
\end{defn}

\noindent The evaluation of a Laurent $F$-crystal $$(\cE, \varphi_{\cE}) \in \text{Vect}((Y, M_Y)_{\Prism}, \cO_{\Prism}[\cI_{\Prism}^{-1}]^\wedge_p)^{\varphi =1} (\text{resp.}\ \text{Vect}((Y_g, M_{Y_g})_{\Prism}, \cO_{\Prism}[\cI_{\Prism}^{-1}]^\wedge_p)^{\varphi =1})$$ on the diagram $$p_1, p_2 : (\fS_L, E, M_{\mathrm{Spf}(\fS_L)}) \rightarrow (\fS_L^{(1)}, E, M_{\mathrm{Spf}(\fS_L^{(1)})})$$ $$(\text{resp.}\ \hspace{10pt} p_1, p_2 : (\fS_{L_g}, E, M_{\mathrm{Spf}(\fS_{L_g})}) \rightarrow (\fS_{L_g}^{(1)}, E, M_{\mathrm{Spf}(\fS_{L_g}^{(1)})}))$$ gives an object in $\textbf{DD}_{\cO_{\cE,L}}$ (resp. $\textbf{DD}_{\cO_{\cE,L_g}}$). In particular, for a Laurent $F$-crystal $(\cE, \varphi_{\cE})$; we have $\cE((\fS_L, E, M_{\mathrm{Spf}(\fS_L)}))=\cM_L$ and $\cE((\fS_{L_g}, E, M_{\mathrm{Spf}(\fS_{L_g})}))=\cM_{L_g}$. 

\begin{thm}
The category $\textbf{DD}_{\cO_{\cE,L}}$ (resp. $\textbf{DD}_{\cO_{\cE,L_g}}$) is equivalent to the category $\text{Rep}^{pr}_{\Z_p}(\Gal_L)$ (resp. $\text{Rep}_{\Z_p}^{pr}(\Gal_{L_g})$) 
\end{thm}

\begin{proof}
See Lem. 3.21 and Prop. 3.24 in \cite{DLMSII} for $\text{Rep}^{pr}_{\Z_p}(\Gal_L)$. See Thm. 5.0.14 in \cite{DL} for $\text{Rep}_{\Z_p}^{pr}(\Gal_{L_g})$). 
\end{proof}

\noindent Indeed, under the above equivalence, for any $T \in \text{Rep}_{\Z_p}(\Gal_L)$, the étale $\varphi$-module attached to the restriction $T|_{\Gal_{L_\infty}}$ is $(\cM_L, \varphi_{\cM_L})$ and for any $T \in \text{Rep}_{\Z_p}(\Gal_{L_g})$, the étale $\varphi$-module attached to the restriction $T|_{\Gal_{L_{\infty,g}}}$ is $(\cM_{L_g}, \varphi_{\cM_{L_g}})$. Now, the image of the inclusion $L_{\infty} \hookrightarrow L_{\infty,g}$ is dense i.e. $\Gal_{L_{\infty}} \cong \Gal_{L_{\infty,g}}$. Incorporating Prop. \ref{Kim-eq}, this is basically Prop. 4.2.5 in \cite{Gao-imp} : the scalar extension functor from $Mod^{\varphi}_{\cO_{\cE,L}}$ to $Mod^{\varphi}_{\cO_{\cE,L_g}}$ is an equivalence. \medskip

\noindent If a $\Z_p$-representation $T$ of $\Gal_L$ (resp. $\Gal_{L_g})$ is of finite $E(u)$-height (with respect to our fixed choice of $\overrightarrow{\pi}$ and $\overrightarrow{X_i}$) then there exists a Breuil-Kisin module $(\fM_L, \varphi_{\fM_L})$ over $\fS_L$ (resp. $(\fM_{L_g}, \varphi_{\fM_{L_g}})$ over $\fS_{L_g}$) of $E$-height $\leq r$ such that the Hodge-Tate weight of $T$ is in $[0, r]$ and the étale $\varphi$-module $\cM_L$ (resp. $\cM_{L_g}$) can be written as $\cM_L=\fM_L \otimes_{\fS_L} \cO_{\cE,L}$ (resp. $\cM_{L_g}=\fM_{L_g} \otimes_{\fS_{L_g}} \cO_{\cE,L_g}$). This motivates to define the full subcategory of finite height descent data inside $\textbf{DD}_{\cO_{\cE,L}}$ (resp. $\textbf{DD}_{\cO_{\cE,L_g}}$) as follows :


\begin{defn}
\noindent A triplet $(\cM_L, \varphi_{\cM_L}, f_{\text{ét}})$ (resp. $(\cM_{L_g}, \varphi_{\cM_{L_g}}, f_{\text{ét},g})$)is called finite height Kisin descent data if there exists a Breuil-Kisin module $(\fM_L, \varphi_{\fM_L})$ (resp. $(\fM_{L_g}, \varphi_{\fM_{L_g}})$) of finite $E$-height such that $\cM_L = \cO_{\cE,L} \otimes_{\fS_L} \fM_L$ (resp. $\cM_{L_g} = \cO_{\cE,L_g} \otimes_{\fS_{L_g}} \fM_{L_g}$). Consequently, $$ f_{\text{ét}} : \fS_L^{(1)}[E^{-1}]^\wedge_p \otimes_{p_1, \fS_L} \fM_L \longrightarrow \fS_L^{(1)}[E^{-1}]^\wedge_p \otimes_{p_2, \fS_L} \fM_L$$  $$(\text{resp.} \hspace{10pt} f_{\text{ét},g} : \fS_{L_g}^{(1)}[E^{-1}]^\wedge_p \otimes_{p_1, \fS_{L_g}} \fM_{L_g} \longrightarrow \fS_{L_g}^{(1)}[E^{-1}]^\wedge_p \otimes_{p_2, \fS_{L_g}} \fM_{L_g})$$
compatible with Frobenii and satisfies cocycle conditions over $\fS_L^{(2)}[E^{-1}]^\wedge_p$ (resp. over $\fS_{L_g}^{(2)}[E^{-1}]^\wedge_p$). Let us denote the full subcategory of finite height Kisin descent data inside $\textbf{DD}_{\cO_{\cE,L}}$ (resp. $\textbf{DD}_{\cO_{\cE,L_g}}$) by $\textbf{DD}^{\text{fin}}_{\cO_{\cE,L}}$ (resp. $\textbf{DD}^{\text{fin}}_{\cO_{\cE,L_g}}$). In particular, an object of $\textbf{DD}^{\text{fin}}_{\cO_{\cE,L}}$ (resp. $\textbf{DD}^{\text{fin}}_{\cO_{\cE,L_g}}$) can be written as a triplet $(\fM_L, \varphi_{\fM_L}, f_{\text{ét}})$ (resp. $(\fM_{L_g}, \varphi_{\fM_{L_g}}, f_{\text{ét},g})$).
\end{defn}

\noindent It is evident that the full subcategory $\text{Rep}^{\text{fin}}_{\Z_p}(G_L)$ (resp. $\text{Rep}^{\text{fin}}_{\Z_p}(G_{L_g})$) is equivalent to the category $\textbf{DD}^{\text{fin}}_{\cO_{\cE,L}}$ (resp. $\textbf{DD}^{\text{fin}}_{\cO_{\cE,L_g}}$). \smallskip

\noindent Therefore, we have the following 2-commutative diagrams with vertical equivalences:

\begin{equation}\label{diag-fin-descent-data}
\begin{tikzcd}
\textbf{DD}^{\text{fin}}_{\cO_{\cE,L}} \arrow[r, hook] \arrow[d, "\cong"] & \textbf{DD}_{\cO_{\cE,L}} \arrow[d, "\cong"] \\
\textbf{Rep}^{\text{fin}}_{\Z_p}(\Gal_L) \arrow[r, hook]           &   \textbf{Rep}_{\Z_p}(\Gal_L)        
\end{tikzcd} 
\hspace{30pt}
\begin{tikzcd}
\textbf{DD}^{\text{fin}}_{\cO_{\cE,L_g}} \arrow[r, hook] \arrow[d, "\cong"] & \textbf{DD}_{\cO_{\cE,L_g}} \arrow[d, "\cong"] \\
\textbf{Rep}^{\text{fin}}_{\Z_p}(\Gal_{L_g}) \arrow[r, hook]           &   \textbf{Rep}_{\Z_p}(\Gal_{L_g})        
\end{tikzcd} 
\end{equation}

\vspace{15pt}

\subsection{Lattices in Semistable Galois Representations:}\label{sec-2-5}

 We now describe the equivalence between the category of $\Z_p$-lattices in semistable $p$-adic representation of $\Gal_L$ (resp. $\Gal_{L_g}$) and the category of \textit{integral Kisin} descent datum over $\fS_L$ (resp. $\fS_{L_g}$) following \cite{DLMSII} (resp. \cite{DL}). 

\begin{defn}
The category $\textbf{DD}_{\fS_L}$ (resp. $\textbf{DD}_{\fS_{L_g}}$) consists of the objects $(\fM_L, \varphi_{\fM_L}, f)$ (resp. $(\fM_{L_g}, \varphi_{\fM_{L_g}}, f_g)$) such that

(1) $(\fM_L, \varphi_{\fM_L})$ (resp. $(\fM_{L_g}, \varphi_{\fM_{L_g}})$) is an object of $\text{Mod}_{\fS_L}^{\varphi}$ (resp. $\text{Mod}_{\fS_{L_g}}^{\varphi}$),

(2) $f$ (resp. $f_g$) is isomorphism:
$$ f : \fS_L^{(1)} \otimes_{\fS_L,\ p_1} \fM_L \rightarrow \fS_L^{(1)} \otimes_{\fS_L,\ p_2} \fM_L $$
$$ (\text{resp.}\ f_g : \fS_{L_g}^{(1)} \otimes_{\fS_{L_g},\ p_1} \fM_{L_g} \rightarrow \fS_{L_g}^{(1)} \otimes_{\fS_{L_g},\ p_2} \fM_{L_g}) $$
over $\fS_L^{(1)}$ (resp. $\fS_{L_g}^{(1)}$) which is compatible with Frobenii and satisfies cocycle condition over $\fS_L^{(2)}$ (resp. $\fS_{L_g}^{(2)}$). The morphisms are $\fS_L$-linear (resp. $\fS_{L_g}$-linear) morphisms compatible with all structures.
\end{defn}

\noindent The following result is Thm. 5.0.18 in \cite{DL} for the CDVF $L_g$ with perfect residue field and Prop. 4.8 in \cite{DLMSII} for the CDVF $L$ with imperfect residue field.

\begin{thm}\label{st-dd}
There is a categorical anti-equivalence: 
$$ \textbf{Rep}^{\text{st}}_{\Z_p}(\Gal_L) \longrightarrow \textbf{DD}_{\fS_L}\  (\text{resp.}\ \textbf{Rep}^{\text{st}}_{\Z_p}(\Gal_{L_g}) \longrightarrow \textbf{DD}_{\fS_{L_g}})$$
such that if a semistable $\Z_p$-lattice $T$ corresponds to the descent data $(\fM_L, \varphi_{\fM_L}, f)$ (resp. $(\fM_{L_g}, \varphi_{\fM_{L_g}}, f_g)$) then $T^\vee = (W(\widehat{\overline{L}}^\flat) \otimes_{\fS_L} \fM_L)^{\varphi=1}$ (resp. $T^\vee = (W(\widehat{\overline{L_g}}^\flat) \otimes_{\fS_{L_g}} \fM_{L_g})^{\varphi=1}$). 
\end{thm}

\noindent The ring extension $\fS_L \rightarrow \cO_{\cE,L}$ gives the functor $\tilde{I}_{\text{ét}}: \text{Mod}_{\fS_L}^{\varphi} \rightarrow \text{Mod}_{\cO_{\cE,L}}^{\varphi}$, which is fully faithful (see arrow (2) of diagram 4.2.2 in \cite{Gao-imp}). The scalar extension along the ring map $\fS_L^{(i)} \rightarrow \fS_L^{(i)}[E^{-1}]^\wedge_p$ for $i=1,2$ induces the functor $I_{\text{ét}} : \textbf{DD}_{\fS_L} \longrightarrow \textbf{DD}_{\cO_{\cE,L}}$. Since the morphisms of $\textbf{DD}_{\fS_L}$ (resp. $\textbf{DD}_{\cO_{\cE,L}}$) are given by $\fS_L$-linear maps (resp. ${\cO_{\cE,L}}$-linear maps) compatible with descent isomorphisms, the fully faithfulness of $\tilde{I}_{\text{ét}}$ implies that $I_{\text{ét}}$ is fully faithful. Moreover, the essential image of $I_{\text{ét}}$ is contained in the full subcategory $\textbf{DD}^{\text{fin}}_{\cO_{\cE,L}}$ of $\textbf{DD}_{\cO_{\cE,L}}$. Therefore, a semistable $\Z_p$-representation of $\Gal_L$ must be of finite $E(u)$-height.\smallskip

\noindent Analogous containment also holds over $\cO_{L_g}$ i.e. $\textbf{DD}_{\fS_{L_g}} \hookrightarrow \textbf{DD}^{\text{fin}}_{\cO_{\cE,L_g}}$. Therefore, a semistable $\Z_p$-representation of $\Gal_{L_g}$ is of finite $E(u)$-height (in fact, with respect to any choice of $\overrightarrow{\pi}$ by Thm. 7.3.1 in \cite{Gao}). Therefore, we have the following $2$-commutative diagrams :


\begin{equation}\label{diag-st-fin}
\begin{tikzcd}
\textbf{DD}_{\fS_L} \arrow[r, hook, "I_{\text{ét}}"] \arrow[d, "\cong"] & \textbf{DD}^{\text{fin}}_{\cO_{\cE,L}} \arrow[d, "\cong"] \\
\textbf{Rep}^{\text{st}}_{\Z_p}(\Gal_L) \arrow[r, hook]           &   \textbf{Rep}^{\text{fin}}_{\Z_p}(\Gal_L)        
\end{tikzcd}
\hspace{20pt}
\begin{tikzcd}
\textbf{DD}_{\fS_{L_g}} \arrow[r, hook, "I_{\text{ét}}"] \arrow[d, "\cong"] & \textbf{DD}^{\text{fin}}_{\cO_{\cE,L_g}} \arrow[d, "\cong"] \\
\textbf{Rep}^{\text{st}}_{\Z_p}(\Gal_{L_g}) \arrow[r, hook]           &   \textbf{Rep}^{\text{fin}}_{\Z_p}(\Gal_{L_g})        
\end{tikzcd}
\end{equation}


\vspace{15pt}

\noindent  

\subsection{Pullback of descent data and restriction of Galois representation}\label{sec-3}

\noindent In this section, we describe the restriction $T|_{\Gal_{L_n}}$ of $\Z_p$-representation $T$ of $\Gal_L$ in terms of the pullback of descent data along the map $\iota_n : \fS_L \rightarrow \fS_{L_n}$ and $\iota_n^{(1)} : \fS_L^{(1)} \rightarrow \fS_{L_n}^{(1)}$ from Subsec. \ref{sec-2-2}. The restriction $T|_{\Gal_{L_g}}$ of $\Z_p$-representation $T$ of $\Gal_L$ via $\cO_L \rightarrow \cO_{L_g}$ corresponds to the pullback of descent data along the map $\iota_g : \fS_L \rightarrow \fS_{L_g}$ and $\iota_g^{(1)} : \fS_L^{(1)} \rightarrow \fS_{L_g}^{(1)}$ from Subsec. \ref{sec-2-3} by Prop. 3.27 (c) in \cite{DLMS}. \medskip

\noindent Recall that $Y=\mathrm{Spf}(\cO_L)$ is equipped with the log structure $M_Y$ given by $\N \rightarrow \cO_L : n \mapsto \pi^n$; so we have the log formal scheme $(Y, M_Y)$. The equivalence $\text{Vect}((Y,M_Y)_{\Prism}, \cO_{\Prism}[\cI_{\Prism}^{-1}]^\wedge_p)^{\varphi=1} \cong \textbf{Rep}_{\Z_p}(\Gal_L)$ is functorial in $(Y, M_Y)$ by Prop. 3.18 in \cite{DLMSII}. In particular, if $(Y_n, M_{Y_n})=(\mathrm{Spf}(\cO_{L_n}), M_{\mathrm{Spf}(\cO_{L_n})})$ and we consider the morphism of sites $\iota_{n, \Prism}: (Y_n, M_{Y_n})_{\Prism} \rightarrow (Y, M_{Y})_{\Prism}$; then the following $2$-commutative diagram holds :

\begin{equation}\label{diag-functoriality}
\begin{tikzcd}
\text{Vect}((Y,M_Y)_{\Prism}, \cO_{\Prism}[\cI_{\Prism}^{-1}]^\wedge_p)^{\varphi=1} \arrow[r, "\cong"] \arrow[d, "\iota_{n, \Prism}^\star"] & \textbf{Rep}_{\Z_p}(\Gal_L) \arrow[d, "\text{restriction}"] \\
\text{Vect}((Y_n,M_{Y_n})_{\Prism}, \cO_{\Prism}[\cI_{\Prism}^{-1}]^\wedge_p)^{\varphi=1} \arrow[r, "\cong"]           &   \textbf{Rep}_{\Z_p}(\Gal_{L_n})       
\end{tikzcd}
\end{equation}

\noindent Here the functor $\iota_{n, \Prism}^\star$ at the left vertical arrow is the pullback functor of Laurent $F$-crystals. 

\begin{prop}\label{prop-functoriality}
Let $T$ be a $\Z_p$-representation of $\Gal_L$ and $(\cM_L, \varphi_{\cM_L}, f_{\text{ét}})$ be the an associated object in $\textbf{DD}_{\cO_{\cE,L}}$. Then the descent data associated to the restriction $T|_{\Gal_{L_n}}$ is given by $(\cM_L, \varphi_{\cM_L}, f_{\text{ét},n})$; where, $f_{\text{ét},n} = \mathrm{id}_{\fS_{L_n}^{(1)}[E_n^{-1}]^\wedge_p} \otimes_{\iota_n^{(1)}, \fS_{L}^{(1)}[E^{-1}]^\wedge_p} f_{\text{ét}}$. In particular, if $T$ is a finite $E(u)$-height $\Z_p$-representation of $\Gal_L$ then its restriction $T|_{\Gal_{L_n}}$ is also of finite $E(u)$-height $\Z_p$-representation of $\Gal_{L_n}$. 
\end{prop}

\begin{proof}
Let $(\cE, \varphi_{\cE})$ be the Laurent $F$-crystal corresponding to the $\Z_p$-representation $T$ of $\Gal_L$. The sheaf $(\cE, \varphi_{\cE})$ on the ringed site $((Y,M_Y)_{\Prism}, \cO_{\Prism}[\cI_{\Prism}^{-1}]^\wedge_p)$ is determined by its evaluation on the diagram $p_1, p_2 : (\fS_L, E, \N)^a \rightarrow (\fS_L^{(1)}, E, \N^2)^a$. Using functoriality i.e. diag \ref{diag-functoriality}, the Laurent $F$-crystal corresponding to $T|_{\Gal_{L_n}}$ is given by $(\iota_{n, \Prism}^\star(\cE), \varphi_{\iota_{n, \Prism}^\star(\cE)}) \in \text{Vect}((Y_n,M_{Y_n})_{\Prism}, \cO_{\Prism}[\cI_{\Prism}^{-1}]^\wedge_p)^{\varphi=1}$. Therefore, the sheaf $(\iota_{n, \Prism}^\star(\cE), \varphi_{\iota_{n, \Prism}^\star(\cE)})$ on the ringed site $((Y_n,M_{Y_n})_{\Prism}, \cO_{\Prism}[\cI_{\Prism}^{-1}]^\wedge_p)$ is determined by its evaluation on the diagram $p_1, p_2 : (\fS_{L_n}, E_n, \N)^a \rightarrow (\fS_{L_n}^{(1)}, E_n, \N^2)^a$. We can regard $(\fS_{L_n}, E_n, \N)^a$ as an object in $(Y, M_Y)_{\Prism}$ via $\cO_L \hookrightarrow \cO_{L_n} \cong \fS_{L_n}/ E_n$. It admits the map of log prisms $\iota_n : (\fS_{L}, E, \N)^a \rightarrow (\fS_{L_n}, E_n, \N)^a$ (cf. Subsec. \ref{sec-2-2}). Then, being crystal in vector bundles over $\cO_{\Prism}[\cI_{\Prism}^{-1}]^\wedge_p$, we have $$\cE((\fS_{L_n}, E_n, \N)^a)=\cE((\fS_{L}, E, \N)^a) \otimes_{\fS_L[E^{-1}]^\wedge_p,\ \iota_n} \fS_{L_n}[E_n^{-1}]^\wedge_p.$$ Therefore, $(\iota_{n, \Prism}^\star(\cE), \varphi_{\iota_{n, \Prism}^\star(\cE)})$ is determined by $\cM_{L_n}:=\cE((\fS_{L_n}, E_n, \N)^a)$ along with descent isomorphism $$\iota_{n, \Prism}^\star(f_{\text{ét}}) : \fS_{L_n}^{(1)}[E_n^{-1}]^\wedge_p \otimes_{p_1, \fS_{L_n}[E_n^{-1}]^\wedge_p} \cM_{L_n} \longrightarrow \fS_{L_n}^{(1)}[E_n^{-1}]^\wedge_p \otimes_{p_2, \fS_{L_n}[E_n^{-1}]^\wedge_p} \cM_{L_n}$$ which is compatible with Frobenii and satisfies cocycle conditions over $ \fS_{L_n}^{(2)}[E_n^{-1}]^\wedge_p$. From the pushout diagram in Subsec. \ref{sec-2-2}, we see that $\iota_{n, \Prism}^\star(f_{\text{ét}}) = \mathrm{id}_{\fS_{L_n}^{(1)}[E_n^{-1}]^\wedge_p} \otimes_{\iota_n^{(1)}, \fS_{L}^{(1)}[E^{-1}]^\wedge_p} f_{\text{ét}}$. 

\end{proof}

\begin{prop}\label{functoriality-imperf-perf}
Let $T$ be a $\Z_p$-representation of $\Gal_L$ and $(\cM_L, \varphi_{\cM_L}, f_{\text{ét}})$ be the an associated object in $\textbf{DD}_{\cO_{\cE,L}}$. Then the descent data associated to the restriction $T|_{\Gal_{L_g}}$ is given by $(\cM_{L_g}, \varphi_{\cM_{L_g}}, f_{\text{ét},g})$; where, $\cM_{L_g}=\fS_{L_g}[E^{-1}]^\wedge_p \otimes_{\iota_g, \fS_L[E^{-1}]^\wedge_p} \cM_{L}$ and $f_{\text{ét},g} = \mathrm{id}_{\fS_{L_g}^{(1)}[E^{-1}]^\wedge_p} \otimes_{\iota_g^{(1)}, \fS_{L}^{(1)}[E^{-1}]^\wedge_p} f_{\text{ét}}$.
\end{prop}

\begin{proof}
This follows from the proof of part (c) of Prop. 3.27 (ii) from \cite{DLMS}. Equivalently, it can be shown similarly as in Prop. \ref{prop-functoriality} using functoriality of Laurent $F$-crystals (and then evaluating those at Breuil-Kisin prisms) along the map $(\mathrm{Spf}(\cO_{L_g}), M_{\mathrm{Spf}(\cO_{L_g})}) \rightarrow (\mathrm{Spf}(\cO_{L}), M_{\mathrm{Spf}(\cO_{L})})$ of log formal schemes. 
\end{proof}

\section{Proof of main result : CDVR case}\label{Sec-4}

\subsection{CDVR version of Thm. \ref{Gao-main-perf}} 

In this section, we prove Thm. \ref{main-thm-2}. This will be used to prove our main result Thm. \ref{main-thm-1} in Sec. \ref{sec-5}. The key ingredient is Prop. \ref{prop-imp-perf-st}.\medskip

\noindent We denote the category of finite $E(u)$-height (with respect to $\overrightarrow{\pi}$ and $\overrightarrow{X_i}$; cf. Def. \ref{def-2}) $\Z_p$-representations of $\Gal_L$ (resp. $\Gal_{L_g}$) by $\text{Rep}_{\Z_p}^{\text{fin}}(\Gal_L)$ (resp. $\text{Rep}_{\Z_p}^{\text{fin}}(\Gal_{L_g})$). The continuous map $\Gal_{L_g} \rightarrow \Gal_L$ induces the restriction functor 
$$ \cT_g : \text{Rep}_{\Z_p}(\Gal_L) \longrightarrow \text{Rep}_{\Z_p}(\Gal_{L_g}) $$
defined by restricting the Galois action : $T \mapsto T|_{\Gal_{L_g}}$. Let $T$ be an object in $\text{Rep}_{\Z_p}^{\text{fin}}(\Gal_L)$ that corresponds to an object $(\fM_L, \varphi_{\fM_L}, f_{\text{ét}})$ in $\textbf{DD}^{\text{fin}}_{\cO_{\cE,L}}$. Then by Prop. \ref{functoriality-imperf-perf}, $\cT_g(T)=T|_{\Gal_{L_g}}$ is an object in $\text{Rep}_{\Z_p}^{\text{fin}}(\Gal_{L_g})$ that corresponds to $(\fM_{L_g}, \varphi_{\fM_{L_g}}, f_{\text{ét},g})$, where the underlying Breuil-Kisin module $\fM_{L_g}=\fM_L \otimes_{\fS_L} \fS_{L_g}$ and $f_{\text{ét},g} = \mathrm{id}_{\fS_{L_g}^{(1)}[E^{-1}]^\wedge_p} \otimes_{\iota_g^{(1)}, \fS_{L}^{(1)}[E^{-1}]^\wedge_p} f_{\text{ét}}$. \smallskip


\noindent Recall that the map of underlying $\delta$-rings for the map $\iota_g : (\fS_L, E, \N)^a \rightarrow (\fS_{L_g}, E, \N)^a$ of prisms in $(Y, M_Y)_{\Prism}^{\text{opp}}$ is given by the classically faithfully flat map $\iota_g : \fS_L \rightarrow \fS_{L_g}$ (cf. Subsec. \ref{sec-2-3}); in particular it is injective. It is not difficult to check that $\fS_L[E^{-1}]^\wedge_p \cap \fS_{L_g} = \fS_L$ as subsets of $\fS_{L_g}[E^{-1}]^\wedge_p$. From the pushout diagram in Subsec. \ref{sec-2-3} we have the map $\iota_g^{(1)} : (\fS^{(1)}_L, E, \N)^a \rightarrow (\fS^{(1)}_{L_g}, E, \N)^a$ which is  a covering map in $(Y, M_Y)_{\Prism}^{\text{opp}}$. It means that the underlying $\delta$-ring map $\iota_g^{(1)} : \fS^{(1)}_L \rightarrow \fS^{(1)}_{L_g}$ is $(p,E)$-completely faithfully flat. This map is not injective; for instance one can see that $\iota_g^{(1)} (X_{i,2}) = \iota_g^{(1)} (X_i)$ for all $1 \leq i \leq b$. This is reminiscent with the fact that the ring $\fS^{(1)}_L$ is derived $(p,E)$-completed and \textit{not} classically $(p,E)$-completed (cf. Sec. 2.3 of \cite{DLMSII}). The map $\iota_g^{(1)} : (\fS^{(1)}_L, E, \N)^a \rightarrow (\fS^{(1)}_{L_g}, E, \N)^a$ of prisms induces the ring map $\iota_g^{(1)} : \fS^{(1)}_L[E^{-1}]^\wedge_p \rightarrow \fS_{L_g}^{(1)}[E^{-1}]^\wedge_p$. We need the following two lemmas about this ring map in the proof of Prop. \ref{prop-imp-perf-st} below.

\begin{lem}\label{intersection-equality}
One has that $\iota_g^{(1)} (\fS^{(1)}_L[E^{-1}]^\wedge_p) \cap \fS^{(1)}_{L_g} = \iota_g^{(1)} (\fS_L^{(1)})$; regarded as subsets of $\fS_{L_g}^{(1)}[E^{-1}]^\wedge_p$.
\end{lem}



\begin{proof}
We have $\fS_L = \cO_{L_0}[[u]]$ and $\fS_{L_g}=W(k_g)[[u]]$ (cf. proof of Prop. \ref{mainthm-prelim}). Recall the description of the rings $\fS_L^{(1)}$ and $\fS_{L_g}^{(1)}$ from Subsec. \ref{sec-2-2} and \ref{sec-2-3}:

$$ \fS_L^{(1)} =  \fS_L[[1-\frac{u_2}{u}, 1-\frac{X_{1,2}}{X_{1}}, 1-\frac{X_{2,2}}{X_{2}}, \cdots, 1-\frac{X_{b,2}}{X_{b}}]]\{ \dfrac{1-\frac{u_2}{u}}{E(u)}, \dfrac{1-\frac{X_{1,2}}{X_{1}}}{E(u)}, \dfrac{1-\frac{X_{2,2}}{X_{2}}}{E(u)} ,\cdots, \dfrac{1-\frac{X_{b,2}}{X_{b}}}{E(u)}  \}^\wedge_{\delta};$$

$$\fS_{L_g}^{(1)} = \fS_{L_g}[[1-\frac{u_2}{u}]]\{ \dfrac{1-\frac{u_2}{u}}{E(u)}\}^\wedge_{\delta}.$$
The map $\iota_g^{(1)} : \fS_L^{(1)} \rightarrow \fS_{L_g}^{(1)}$ is given by $X_i \mapsto [X_i^\flat]$ and $X_{i,2} \mapsto [X_i^\flat]$ for $1 \leq i \leq b$. Let us write \footnote{We denote it by $\fS^{(1)}_{K}$ because it is self-coproduct of Breuil-Kisin log prism $(\fS_K, E(u), \N)^a$ over $\cO_K$; where $\fS_K=W(k)[[u]]$.} $\fS^{(1)}_{K} := W(k)[[u]][[1-\frac{u_2}{u}]]\{ \dfrac{1-\frac{u_2}{u}}{E(u)}\}^\wedge_{\delta}$ and $A:=\iota_g(\cO_{L_0}) \subset W(k_g)$. Then, the image $\iota_g^{(1)}(\fS_L^{(1)})$ is isomorphic to $\fS^{(1)}_{K} \widehat{\otimes}_{W(k)} A$. Similarly, $\fS_{L_g}^{(1)} \cong \fS^{(1)}_{K} \widehat{\otimes}_{W(k)} W(k_g)$. Also, note that the quotient $W(k_g)/A$ is $p$-torsion free since the map of corresponding residue fields $k(\overline{X_1}, \dots, \overline{X_b}) \hookrightarrow k_g$ is injective.\smallskip


\noindent Now, the ring $\fS^{(1)}_{K}$ is $p$-completely flat over $W(k)$. So, we apply $\fS^{(1)}_{K} \widehat{\otimes}_{W(k)} -$ to the exact sequence $$ 0 \rightarrow A \rightarrow W(k_g) \rightarrow W(k_g)/A \rightarrow 0$$ to get the short exact sequence: 
$$ 0 \rightarrow \iota_g^{(1)}(\fS_L^{(1)}) \rightarrow \fS_{L_g}^{(1)} \rightarrow Q \rightarrow 0;$$ where $Q := \fS^{(1)}_{K} \widehat{\otimes}_{W(k)} (W(k_g)/A)$. We have $Q/pQ = (\fS^{(1)}_{K} / p.\fS^{(1)}_{K}) \otimes_k (k_g / k(\overline{X_1}, \overline{X_2}, \cdots, \overline{X_b}))$. Here, $k_g / k(\overline{X_1}, \overline{X_2}, \cdots, \overline{X_b})$ is regarded as quotient vector space over $k$. Note that $u$ is a non-zero divisor in $Q/pQ$ and $E(u)=u^e$ in $Q/pQ$; so $E(u)$ is non-zero divisor in $Q/pQ$. Since $Q$ is $p$-torsion free we can use induction to conclude that $E(u)$ is non-zero divisor in $Q/p^nQ$ for all $n \geq 1$. Therefore, $Q \hookrightarrow Q[E^{-1}]^\wedge_p$ is injective as $Q$ is $p$-adically complete. Since localization is exact and inverse limit is left-exact, we have
\begin{equation}
\begin{tikzcd}
0 \arrow[r] & \iota_g^{(1)}(\fS_L^{(1)}) \arrow[r] \arrow[d] & \fS_{L_g}^{(1)} \arrow[r] \arrow[d] & Q \arrow[d, "\text{injective}"] \\
0 \arrow[r] & \iota_g^{(1)}(\fS_L^{(1)})[E^{-1}]^\wedge_p \arrow[r]           & \fS_{L_g}^{(1)}[E^{-1}]^\wedge_p \arrow[r]           & Q[E^{-1}]^\wedge_p          
\end{tikzcd}
\end{equation}

\noindent Let $x \in \iota_g^{(1)}(\fS_L^{(1)})[E^{-1}]^\wedge_p \cap \fS_{L_g}^{(1)} \subset \fS_{L_g}^{(1)}[E^{-1}]^\wedge_p$. Then, $x$ mapsto $0$ in $Q[E^{-1}]^\wedge_p$. By injectivity of $Q \hookrightarrow Q[E^{-1}]^\wedge_p$ and exactness of the upper sequence, $x \in \iota_g^{(1)}(\fS_L^{(1)})$. The reverse inclusion is obvious. Hence, we have the equality  $\iota_g^{(1)}(\fS_L^{(1)})[E^{-1}]^\wedge_p \cap \fS_{L_g}^{(1)} = \iota_g^{(1)}(\fS_L^{(1)})$.
\end{proof}


\begin{lem}\label{kernel-lemma}
The kernel of the induced map $\iota_g^{(1)} : \fS_L^{(1)}[E^{-1}]^\wedge_p \rightarrow \fS_{L_g}^{(1)}[E^{-1}]^\wedge_p$ is contained in $\fS_L^{(1)}$. 
\end{lem}

\begin{proof}
Let us denote the kernel of the map (before taking $p$-adic completion) $\iota_g^{(1)} : \fS_L^{(1)}[E^{-1}] \rightarrow \fS_{L_g}^{(1)}[E^{-1}]$ by $K$. We claim that it is enough to show $K \subset \fS_L^{(1)}$. Suppose that $K \subset \fS_L^{(1)}$. We have the exact sequence $$ 0 \rightarrow K \rightarrow \fS_L^{(1)}[E^{-1}] \rightarrow \fS_{L_g}^{(1)}[E^{-1}].$$  From the proof of the Lem. \ref{intersection-equality}, $\dfrac{\fS_{L_g}^{(1)}}{\iota_g^{(1)}(\fS_L^{(1)})}$ is $p$-torsion free (as it is a subset of $Q$ which is $p$-torsion free). So, $\dfrac{\fS_{L_g}^{(1)}}{ \iota_g^{(1)}(\fS_L^{(1)})}[E^{-1}] = \dfrac{\fS_{L_g}^{(1)}[E^{-1}]}{\iota_g^{(1)}(\fS_L^{(1)})[E^{-1}]}$ is $p$-torsion free and therefore $p^n$-torsion free for all $n$. Similarly, $\dfrac{\fS_L^{(1)}[E^{-1}]}{K}$ (regarded as subset of $\fS_{L_g}^{(1)}[E^{-1}]$) is $p^n$-torsion free. So for every $n$, we have an exact sequence $$ 0 \rightarrow \frac{K}{p^n} \rightarrow \dfrac{\fS_L^{(1)}[E^{-1}]}{ p^n} \rightarrow \dfrac{\fS_{L_g}^{(1)}[E^{-1}]}{p^n}$$ which gives us an inverse system of exact sequences. Since $\fS_L^{(1)}[E^{-1}], \fS_{L_g}^{(1)}[E^{-1}]$ are $p$-adically separated and inverse limit is a left exact functor, we have $$ 0 \rightarrow \varprojlim\limits_n \frac{K}{p^n} \rightarrow \fS_L^{(1)}[E^{-1}]^\wedge_p \rightarrow \fS_{L_g}^{(1)}[E^{-1}]^\wedge_p.$$ Now, by our assumption $K \subset \fS_L^{(1)}$ and $\fS_L^{(1)}$ is classically $p$-completed. Therefore, $\varprojlim\limits_n \dfrac{K}{p^n} \subset \fS_L^{(1)}$.\smallskip

\noindent So, we are left to show that $K \subset \fS_L^{(1)}$. Let $ \dfrac{x}{E^k} \in \fS_L^{(1)}[E^{-1}]$ such that $\dfrac{\iota_g^{(1)}(x)}{E^k}=0$ in $\fS_{L_g}^{(1)}[E^{-1}]$. Since, $\fS_{L_g}^{(1)}$ is $E$-torsion free, we have $\iota_g^{(1)}(x)=0$. The map $\iota_g^{(1)} : \fS_L^{(1)} \rightarrow \fS_{L_g}^{(1)}$ is $(p, E)$-completely faithfully flat (cf. Sec. \ref{sec-2-2}). Therefore, for each $n,m \in \N$ the reduction map $\iota_{g,n,m}^{(1)} : \dfrac{\fS_L^{(1)}}{(p^n, E^m)} \rightarrow  \dfrac{\fS_{L_g}^{(1)}}{(p^n, E^m)}$ is classically faithfully flat; in particular it is injective. This implies, $x \in \cap_{n,m} (p^n , E^m)\fS_L^{(1)}$.\smallskip

\noindent Now fix $m=k+1$. Then for each $n \in \N$, $x = p^n x_n + E^m z_n$ for some $x_n, z_n \in \fS_L^{(1)}$. So, $\dfrac{x}{E^k} = p^n \dfrac{x_n}{E^k} + E.z_n$. Let us write $y_n = E.z_n$ for each $n$. We have $\dfrac{x}{E^k} \equiv y_n \pmod{p^n \fS_L^{(1)}[E^{-1}]}$. The difference $y_{n+1} - y_n = p^n (\dfrac{x_n}{E^k} - p \dfrac{x_{n+1}}{E^k})$. Write $\gamma_n = \dfrac{x_n - p.x_{n+1}}{E^k} \in \fS_L^{(1)}[E^{-1}]$. But $y_{n+1} - y_n \in \fS_L^{(1)}$, so $\gamma_n \in \fS_L^{(1)}[E^{-1}] \cap \fS_L^{(1)}[p^{-1}] = \fS^{(1)}_L$ (the last equality is due to Lem. 2.17 in \cite{DLMSII}). Therefore, $y_{n+1} - y_n \in p^n \fS_L^{(1)}$. Consequently, it gives us an element $(y_n)_n$ of the inverse limit $\varprojlim\limits_n \dfrac{\fS_L^{(1)}}{p^n \fS_L^{(1)}}$. Since, $\fS_L^{(1)}$ is classically $p$-completed, we get an element $y \in \fS_L^{(1)}$.\smallskip

\noindent Finally, the natural map $\fS_L^{(1)} \hookrightarrow \fS_L^{(1)} [E^{-1}]$ is injective and both side are $p$-adically separated (by Lem. 2.17 in \cite{DLMSII}). Hence, $\dfrac{x}{E^k} \in \fS_L^{(1)}$. This completes the proof.
\end{proof}


\begin{prop}\label{prop-imp-perf-st}
Let $V$ be finite $E(u)$-height $p$-adic representation of $\Gal_L$. If its restriction $V|_{\Gal_{L_g}}$ is semistable $p$-adic representation of $\Gal_{L_g}$, then $V$ is semistable $p$-adic representation of $\Gal_L$.
\end{prop}

\begin{proof}
Since $V$ is finite height $p$-adic representation, there exists a $\Gal_L$-stable $\Z_p$-lattice $T$ in $V$ and finite height Breuil-Kisin module $(\fM_L, \varphi_{\fM_L})$ over $\fS_L$ such that the descent data associated to $T$ is given by $(\fM_L, \varphi_{\fM_L}, f_{\text{ét}})$; where the descent isomorphism $$ f_{\text{ét}} : \fS_L^{(1)}[E^{-1}]^\wedge_p \otimes_{p_1, \fS_L} \fM_L \rightarrow \fS_L^{(1)}[E^{-1}]^\wedge_p \otimes_{p_2, \fS_L} \fM_L$$ compatible with Frobenii and satisfies cocycle conditions over $\fS_L^{(2)}[E^{-1}]^\wedge_p$.\smallskip


\noindent Since $T|_{\Gal_{L_g}}$ is a $\Gal_{L_g}$-stable $\Z_p$-lattice in semistable $p$-adic representation of $\Gal_{L_g}$, by Thm. \ref{st-dd}, it corresponds to integral Kisin descent data $(\fM_{L_g}, \varphi_{\fM_{L_g}}, f_g)$ in $\textbf{DD}_{\fS_{L_g}}^{\text{fin}}$ given by $$ f_g : \fS_{L_g}^{(1)} \otimes_{p_1, \fS_{L_g}} \fM_{L_g} \rightarrow \fS_{L_g}^{(1)}\otimes_{p_2, \fS_{L_g}} \fM_{L_g}.$$ By Prop. \ref{functoriality-imperf-perf}, the descent data $$(\fS_{L_g} \otimes_{\fS_L} \fM_L, \varphi \otimes \varphi_{\fM_L}, \mathrm{id}_{\fS^{(1)}_{L_g}[E^{-1}]^\wedge_p} \otimes_{\iota^{(1)}_g, \fS_L^{(1)}[E^{-1}]^\wedge_p} f_{\text{ét}})\ \text{and}\  (\fM_{L_g}, \varphi_{\fM_{L_g}}, \mathrm{id}_{\fS^{(1)}_{L_g}[E^{-1}]^\wedge_p} \otimes_{\fS_{L_g}^{(1)}} f_{g})$$ are isomorphic in the category $\textbf{DD}^{\text{fin}}_{\cO_{\cE,L_g}}$. In particular, $\fM_{L_g} \cong \fS_{L_g} \otimes_{\fS_L} \fM_L$ as Breuil-Kisin modules over $\fS_{L_g}$. \smallskip

\noindent Let us fix a basis of $\fM_L$. We denote the matrix of the isomorphism $f_{\text{ét}}$ by $M \in \text{Mat}(\fS_L^{(1)}[E^{-1}]^\wedge_p)$. Then, the matrix of $\mathrm{id}_{\fS^{(1)}_{L_g}[E^{-1}]^\wedge_p} \otimes_{\iota^{(1)}_g, \fS_L^{(1)}[E^{-1}]^\wedge_p} f_{\text{ét}}$ is given by $\iota_g^{(1)}(M) \in \text{Mat}(\iota_g^{(1)}(\fS_L^{(1)}[E^{-1}]^\wedge_p))$. On the other hand, the matrix of $\mathrm{id}_{\fS^{(1)}_{L_g}[E^{-1}]^\wedge_p} \otimes_{\fS_{L_g}^{(1)}} f_g$ belongs to $\text{Mat}(\fS_{L_g}^{(1)})$. By Lem. \ref{intersection-equality}, we have $\iota_g^{(1)}(\fS^{(1)}_L) = \iota_g^{(1)}(\fS^{(1)}_L[E^{-1}]^\wedge_p) \cap \fS_{L_g}^{(1)} \subset \fS^{(1)}_{L_g}[E^{-1}]^\wedge_p$. So, there exists a matrix $N \in \text{Mat}(\fS_L^{(1)})$ such that $\iota_g^{(1)}(M)= \iota_g^{(1)}(N)$. Now the Lem. \ref{kernel-lemma} implies that all the entries of the matrix $M$ belong to $\fS_L^{(1)}$. Therefore, $f_{\text{ét}}$ is the scalar extension of an isomorphism $$ f : \fS_L^{(1)} \otimes_{p_1, \fS_L} \fM_L \rightarrow \fS_L^{(1)} \otimes_{p_2, \fS_L} \fM_L.$$ The map $\iota_g^{(1)}$ is compatible with Frobenius (as it is induced from a map of prisms), so $f$ is compatible with Frobenii. Clearly, it satisfies cocycle conditions over $\fS_L^{(2)}$.\smallskip

\noindent Therefore, we get the integral Kisin descent data $(\fM_L, \varphi_{\fM_L}, f)$ in $\textbf{DD}_{\fS_L}$ that is associated to $T$. Hence, $V=T \otimes_{\Z_p} \Q_p$ is a semistable $p$-adic representation of $\Gal_L$.
\end{proof}

\noindent Now, it is easy to deduce our main result in CDVR case from the above Prop. \ref{prop-imp-perf-st}. Recall that $K^{un} \subset \overline{K}$ is the maximal unramified extension of $K$. For a non-negative integer $m$, we write $L_m=L(\pi_m)$. 

\begin{thm}\label{mainthm-cdvf}
Let $m=\text{max} \{ i \in \N \cup \{0\} : \zeta_{p^i} \in K^{un} \}$. Let $V$ be a finite $E(u)$-height $p$-adic representation of $\Gal_L$. Then $V|_{\Gal_{L_m}}$ is semistable $p$-adic representation of $\Gal_{L_m}$. 
\end{thm}

\begin{proof}

\noindent The action of $\Gal_{L_g}$ on $\overline{K}$ factors through the continuous map $\Gal_{L_g} \rightarrow \Gal_K$. In particular, $\zeta_{p^i} \in K^{ur} \Leftrightarrow\ \text{the inertia subgroup}\ I_K \subset \Gal_K\ \text{acts trivially on}\ \zeta_{p^i} \Leftrightarrow\ \text{the inertia subgroup}\ I_{L_g} \subset \Gal_{L_g}\ \text{acts trivially on}\ \zeta_{p^i} \Leftrightarrow \zeta_{p^i} \in L_g^{ur}$; where $L_g^{ur}$ is the maximal unramified extension of $L_g$ inside a fixed algebraic closure $\overline{L_g}$.\smallskip

\noindent $V$ is finite $E(u)$-height $p$-adic representation of $\Gal_L$ implies $V|_{\Gal_{L_m}}$ (as $p$-adic representation of $\Gal_{L_m}$) and $V|_{\Gal_{L_g}}$ (as $p$-adic representation of $\Gal_{L_g}$) are so by Prop. \ref{prop-functoriality} and \ref{functoriality-imperf-perf}, respectively.  By Thm. \ref{Gao-main-perf}, $V|_{\Gal_{L_g(\pi_m)}}$ is semistable $\Z_p$-representation of $\Gal_{L_g(\pi_m)}$. Clearly, $(L_m)_g = L_g(\pi_m)$. Therefore, $V|_{\Gal_{L_m}}$ is a finite $E(u)$-height $p$-adic representation of $\Gal_{L_m}$ and its restriction $V|_{\Gal_{(L_m)_g}}$ is semistable representation of $\Gal_{(L_m)_g}$. Hence, by Prop. \ref{prop-imp-perf-st}, $V|_{\Gal_{L_m}}$ is semistable $p$-adic representation of $\Gal_{L_m}$. 
\end{proof}

\noindent As mentioned in Subsec. \ref{subsec-1.2}, Thm. \ref{mainthm-cdvf} does not fully provide the CDVR version of Thm. \ref{Gao-main-perf}. Indeed, it does not prove that $V|_{\Gal_{L_m}}$ can be extended to a semistable $p$-adic representation of $\Gal_L$. An essential input to prove Thm. \ref{Gao-main-perf} in \cite{Gao} is the overconvergence of (rational) étale $(\varphi, \tau)$-modules for any CDVR with \textit{perfect} residue field (see Thm. 3.2.2 in \textit{loc. cit.}). In our imperfect residue field case, we need the overconvergence of Laurent $F$-crystals or equivalently its associated descent datum. This has been conjectured as Conj. 2.24 in \cite{DLo}. All of these would be discussed in a sequel of this work.

\section{Proof of main result : small affine case}\label{sec-5}

In this section, we prove the potential semistability of (geometric) families of \textit{finite height representations} in small affine case using the purity result Thm. \ref{purity-thm} from \cite{DLMSII}. We define the notion of finite $E(u)$-height for a $\Z_p$-local system on a smooth affinoid adic space $\cX$ over $K$ having semistable reduction. We see that a finite $E(u)$-height $\Z_p$-local system over $\cX$ becomes \textit{semistable} after pullback along a finite étale cover of $\cX$.\smallskip

\subsection{Finite $E(u)$-height $\Z_p$-local system}\label{sec-5-1}
\noindent We fix a \textit{connected} $\cO_K$-algebra $R$ equipped with a $p$-adically completed étale map $\square : R^0 \rightarrow R$; where $$R^0 := \cO_K \langle X_1, X_2, \cdots, X_r, X_{r+1}^{\pm}, X_{r+2}^{\pm}, \cdots X_{b}^{\pm} \rangle/ (X_1.X_2. \dots X_r - \pi).$$
So, $X:=\mathrm{Spf}(R)$ is a semistable affine formal scheme over $\cO_K$. The adic generic fibre $\cX$ of $X$ is identified with locally noetherian adic space $\cX = \mathrm{Spa}(R[1/p], R)$. Let us denote the category of $\Z_p$-local systems over $\cX$ by $\textbf{Loc}_{\Z_p}(\cX)$. The formal scheme $X$ is equipped with the log structure $M_X$ given by $\N^r \rightarrow R^0 \xrightarrow{\square} R : e_\ell \mapsto X_\ell$. So, we have a log formal scheme $(X,M_X)$. Let us denote the corresponding absolute log prismatic site by $(X, M_X)_{\Prism}$.\smallskip

\noindent We need to define the category of descent data which is equivalent to $\textbf{Loc}_{\Z_p}(\cX)$. Let us briefly recall the Breuil-Kisin prism in $(X, M_X)_{\Prism}$ from Example \ref{example}. We have $$\fS_{R^0} = W(k) \langle X_1, X_2, \cdots, X_r, X_{r+1}^{\pm}, X_{r+2}^{\pm}, \cdots X_{b}^{\pm} \rangle [[u]]/ (X_1.X_2. \dots X_r - u).$$ By Rem. 4.2 in \cite{KosI}, the $p$-adically completed étale map $\square: R^0 \rightarrow R$ induces $(p, E(u))$-adically completed étale map $\square_{\fS} : \fS_{R^0} \rightarrow \fS_{R,\square}$. We write $\fS_R :=\fS_{R,\square}$ (denoted by $\fS$ in Example \ref{example}). So, we have the Breuil-Kisin log prism $(\mathrm{Spf}(\fS_R), E(u), M_{\mathrm{Spf}(\fS_R)})$ in $(X, M_X)_{\Prism}$. Let us denote the self coproduct and self triple-coproduct of $(\mathrm{Spf}(\fS_R), E(u), M_{\mathrm{Spf}(\fS_R)})$ by $(\mathrm{Spf}(\fS_R^{(1)}), E(u), M_{\mathrm{Spf}(\fS_R^{(1)})})$ and $(\mathrm{Spf}(\fS_R^{(2)}), E(u), M_{\mathrm{Spf}(\fS_R^{(2)})})$, respectively in $(X, M_X)^{\text{opp}}_{\Prism}$. We refer the reader to Subsec. 2.2 in \cite{DLMSII} for precise description of these coproducts. These are equipped with projection maps given by $p_1, p_2 : \fS_R \rightarrow \fS_R^{(1)}$ and $p_{12},p_{23},p_{13} : \fS_R^{(1)} \rightarrow \fS_R^{(2)}$. Write $\cO_{\cE,R}=\fS_R[E^{-1}]^\wedge_p$.

\begin{defn}
We define the category $\textbf{DD}_{\cO_{\cE,R}}$ of descent data over $\cO_{\cE,R}$ consisting of triplets $(\cM_R, \varphi_{\cM_R}, f_{\text{ét}})$, where

\begin{itemize}
\item $(\cM_R, \varphi_{\cM_R})$ is étale $\varphi$-module over $\cO_{\cE,R}$ i.e. $\cM_R$ is projective $\cO_{\cE,R}$-module equipped with $\varphi$-semilinear endomorphism $\varphi_{\cM_R}$ whose $\varphi$-linearization is an isomorphism, 

\item $f$ is an isomorphism : $$ f_{\text{ét}} : \fS_R^{(1)}[E^{-1}]^\wedge_p \otimes_{p_1, \fS_R[E^{-1}]^\wedge_p} \cM_R \rightarrow \fS_R^{(1)}[E^{-1}]^\wedge_p \otimes_{p_2, \fS_R[E^{-1}]^\wedge_p} \cM_R$$ compatible with Frobenii and satisfies cocycle conditions over $\fS_R^{(2)}[E^{-1}]^\wedge_p$. 
\end{itemize}
The morphisms are $\cO_{\cE,R}$-linear maps compatible with all structures.
\end{defn}

\noindent Recall the equivalences in the diagram \ref{triangle-diag}. More precisely we have the following:
\begin{thm}(Thm. 3.14 and Lem. 3.21 in \cite{DLMSII})\label{cat-eq-rel}
The categories $\textbf{Loc}_{\Z_p}(\cX)$ and $\textbf{DD}_{\cO_{\cE,R}}$ are equivalent. 
\end{thm}

\noindent We use this equivalence to define finite $E$-height $\Z_p$-local system $\bL$ over $\cX$ as in Def. \ref{finiteEheightlocal system-def}.

\begin{defn}\label{BreuilKisin-relative}
\begin{itemize}
\item A pair $(\fM_R, \varphi_{\fM_R})$ is called a Breuil-Kisin module over $\fS_R$ of $E$-height $\leq r$ if $\fM_R$ is a finite torsion free $\fS_R$-module such that $\fM_R[p^{-1}]$ (resp. $\fM_R[E^{-1}]$) is projective $\fS_R[p^{-1}]$-module (resp. $\fS_R[E^{-1}]$-module), $\fM_R = \fM_R[p^{-1}] \cap \fM_R[E^{-1}]$ and $\varphi_{\fM_R}$ is $\varphi$-semilinear endomorphism such that the cokernel of its $\varphi$-linearization $\varphi_{\fM_R}^\star$ is killed by $E(u)^r$.\smallskip

\item A $\Z_p$-local system $\bL$ is called of finite $E$-height if the underlying étale $\varphi$-module $(\cM_R, \varphi_{\cM_R})$ of the descent data $(\cM_R, \varphi_{\cM_R}, f_{\text{ét}})$ attached to $\bL$ arises from a Breuil-Kisin module $(\fM_R, \varphi_{\fM_R})$ over $\fS_R$ of finite $E$-height i.e. $\cM_R = \cO_{\cE,R} \otimes_{\fS_R} \fM_R$.
\end{itemize}
\end{defn}

\noindent If $\bL$ is of finite $E$-height, then the associated descent data $(\cM_R, \varphi_{\cM_R}, f_{\text{ét}})$ can be written as $(\fM_R, \varphi_{\fM_R}, f_{\text{ét}})$, where $f_{\text{ét}}$ is an isomorphism: $$ f_{\text{ét}} :  \fS_R^{(1)}[E^{-1}]^\wedge_p \otimes_{p_1, \fS_R} \fM_R \rightarrow \fS_R^{(1)}[E^{-1}]^\wedge_p \otimes_{p_2, \fS_R} \fM_R$$ compatible with Frobenii and satisfies cocycle conditions over $\fS_R^{(2)}[E^{-1}]^\wedge_p$. We denote the full subcategory of objects $(\fM_R, \varphi_{\fM_R}, f_{\text{ét}})$ inside $\textbf{DD}_{\cO_{\cE,R}}$ by $\textbf{DD}^{\text{fin}}_{\cO_{\cE,R}}$. Therefore, the above categorical equivalence i.e. Thm. \ref{cat-eq-rel} restricts between the category $\textbf{Loc}^{\text{fin}}_{\Z_p}(\cX)$ of finite $E$-height $\Z_p$-local systems (in the sense Def. \ref{BreuilKisin-relative}) and the category $\textbf{DD}^{\text{fin}}_{\cO_{\cE,R}}$ of finite $E$-height descent data over $\fS_R^{(1)}[E^{-1}]^\wedge_p$. In summary, we have the following $2$-commutative diagram : 

\begin{equation}\label{diag-5-1}
\begin{tikzcd}
\textbf{DD}^{\text{fin}}_{\cO_{\cE,R}} \arrow[r, hook] \arrow[d, "\cong"] & \textbf{DD}_{\cO_{\cE,R}} \arrow[d, "\cong"] \\
\textbf{Loc}^{\text{fin}}_{\Z_p}(\cX) \arrow[r, hook]           &   \textbf{Loc}_{\Z_p}(\cX)       
\end{tikzcd}
\end{equation}

\subsection{Pullback of $\Z_p$-local system and Descent data:}\label{sec-5-2}

Let us write $X_n^0 := \mathrm{Spf}(R_n^0)$, where $$R_n^0 := \cO_{K_n} \langle X_1, X_2, \cdots, X_r, X_{r+1}^{\pm}, X_{r+2}^{\pm}, \cdots X_{b}^{\pm} \rangle/ (X_1.X_2. \dots X_r - \pi_n).$$
We have the map $\iota_n^0 : R^0 \rightarrow R_n^0 $ given by $X_i \mapsto X_i^{p^n}$ for all $1 \leq i \leq r$ and $X_j \mapsto X_j$ for all $r+1 \leq j \leq b$. Consider the $p$-completed tensor product $R_n:=R_n^0 \widehat{\otimes}_{R^0} R$; where $R_n^0$ is $R^0$-algebra via $\iota_n^0$. Write $X_n:=\mathrm{Spf}(R_n)$; so we have $X_n = X_n^0 \times_{X^0} X$ in the category of formal scheme over $\mathrm{Spf}(\cO_K)$. It is equipped with a map of semistable formal schemes $\iota_n : X_n \rightarrow X$. Pulling back the natural log structure we get a map of log formal schemes $(X_n, M_{X_n}) \rightarrow (X, M_X)$. Moreover, this gives rise to a map of absolute log prismatic sites $ \iota_{n,\Prism} : (X_n, M_{X_n})_{\Prism} \rightarrow (X, M_X)_{\Prism}$.\smallskip

\noindent Denote the generic fibre of $X_n$ by $\cX_n$. We have the map of adic spaces $\iota_n : \cX_n \rightarrow \cX$. Let $\bL$ be a $\Z_p$-local system over $\cX$ and $(\cM_R, \varphi_{\cM_R}, f_{\text{ét}})$ be its associated descent data in $\textbf{DD}_{\cO_{\cE,R}}$. The pullback of $\bL$ along the map $\iota_n$ gives a $\Z_p$-local system $\bL_n :=\iota_n^\star  \bL$ over $\cX_n$. By functoriality of Thm. \ref{cat-eq-rel}, we can describe the descent data associated to $\bL_n$ by pulling back the descent data $(\cM_R, \varphi_{\cM_R}, f)$. Write $$ \fS_{R_n}^0 : = \fS_R^0 = W(k) \langle X_1, X_2, \cdots, X_r, X_{r+1}^{\pm}, X_{r+2}^{\pm}, \cdots X_{b}^{\pm} \rangle [[u]]/ (X_1.X_2. \dots X_r - u) $$ and we have the map of prisms $\iota_n^0 : (\fS_{R^0}, E, \N^r)^a \rightarrow (\fS_{R_n^0}, E_n, \N^r)^a$ defined by $u \mapsto u^{p^n}$ and $X_i \mapsto X_i^{p^n}$ for all $1 \leq i \leq r$. Now, we take the coproduct of $(\fS_R, E, M_{\mathrm{Spf}(\fS_R)})^a$ and  $(\fS_{R_n^0}, E_n, \N^r)^a$ over $(\fS_{R^0}, E, \N^r)^a$ in $(X,M_X)^{\text{opp}}_{\Prism}$ and denote it by $(\fS_{R_n}, E_n, M_{\mathrm{Spf}(\fS_{R_n})})^a$. It is equipped with the map of prisms $\iota_n : (\fS_R, E, M_{\mathrm{Spf}(\fS_R)})^a \rightarrow (\fS_{R_n}, E_n, M_{\mathrm{Spf}(\fS_{R_n})})^a$. Similar to the pushout diagram \ref{pushout-sec-2-2}, we have a map of prisms $$\iota_n^{(i)} : (\fS_R^{(i)}, E, M_{\mathrm{Spf}(\fS_R^{(i)})})^a \rightarrow (\fS_{R_n}^{(i)}, E_n, M_{\mathrm{Spf}(\fS_{R_n}^{(i)})})^a \hspace{20pt}\ \text{for}\ i =1,2.$$ 

\noindent Similar to the discussion in Subsec. \ref{sec-3}, functoriality of the equivalence in Thm. \ref{cat-eq-rel}, the descent data associated to $\bL_n$ is given by $(\cM_{R}, \varphi_{\cM_{R}}, f_{\text{ét},n})$; where $f_{\text{ét},n}=\mathrm{id}_{\fS_{R_n}^{(1)}[E^{-1}]^\wedge_p} \otimes_{\iota_n^{(1)}, \fS_R^{(1)}[E^{-1}]^\wedge_p} f_{\text{ét}}$. In particular, if $\bL$ is a finite $E$-height $\Z_p$-local system over $\cX$ with the associated descent data $(\fM_R, \varphi_{\fM_R}, f_{\text{ét}})$ then its pullback $\bL_n$ over $\cX_n$ is also of finite $E$-height as it corresponds to the descent data $(\fM_{R}, \varphi_{\fM_{R}}, f_{\text{ét},n})$.\medskip

\noindent We restrict ourselves to the assumption that $\square : R^0 \rightarrow R$ induces bijection between the set of generic points of $X^0=\mathrm{Spf}(R^0)$ and $X=\mathrm{Spf}(R)$. Let $\{ \xi_1, \xi_2, \cdots, \xi_r\}$ be the set of generic points of all irreducible components of $X$. These points are given by $\{ (X_1), (X_2), \cdots, (X_r) \}$. For each $\xi_j$ we have the complete discrete valuation ring $\cO_{X, \xi_j}^\wedge=R_{(X_j)}^\wedge$ with uniformizer $\pi$ for all $1\leq j \leq r$. Write $\cO_{L_j}$ for $R_{(X_j)}^\wedge$ and $L_j$ for its fraction field. Let $\Delta_j := \mathrm{Spf}(\cO_{X, \xi_j})$. The localization map $ \iota_j : R \rightarrow  \cO_{L_j}$ induces the morphism of log formal schemes $ \iota_j : (\Delta_j, M_{\Delta_j}) \rightarrow (X, M_X)$, where $M_{\Delta_j}$ is the pullback log structure from $M_X$. This moreover gives the map of sites $\iota_{j, \Prism} : (\Delta_j, M_{\Delta_j})_{\Prism} \rightarrow (X, M_X)_{\Prism}$. For each $j$, choose a Cohen ring $\cO_{L_{0,j}} \subset \cO_{L_j}$ of $\text{Frac}((\cO_{L_j})/(\pi))$ as a subring of $\cO_{L_j}$ with a lift of Frobenius $\varphi$ such that $\varphi(X_i)=X_i^p$ for $i \neq j$. Clearly, $\cO_{L_j} = \cO_{L_{0,j}}[\pi]$. Consider the ring $\fS_{L_j} := \cO_{L_{0,j}}[[u]]$ equipped with obvious Frobenius given by $u \mapsto u^p$ and $X_i \mapsto X_i^p$ for $i \neq j$. This gives a Breuil-Kisin log prism $(\fS_{L_j}, E(u), \N)^a$ in the site $(\Delta_j, M_{\Delta_j})_{\Prism}$. For each $j$, we have
$$ \iota_j^0 : \fS_{R^0} \rightarrow \fS_{L_j}\ \text{given by}\ X_j \mapsto u(X_1.X_2 \dots X_{j-1}X_{j+1} \dots X_r)^{-1}; u \mapsto u; X_i \mapsto X_i\ (i \neq j).$$ This lifts uniquely to $\iota_j: \fS_R \rightarrow  \fS_{L_j}$, which is flat, injective and compatible with Frobenius. Alternatively, one can consider the map $ \iota_j : R \rightarrow  \cO_{L_j}$ which is $p$-completely étale and use Rem. 4.2 of \cite{KosI} to obtain a map of log prisms $\iota_j : (\fS_R, E, \N^r)^a \rightarrow (\fS_{L_j}, E, \N)^a$ in $(X, M_X)_{\Prism}^{\text{opp}}$. Indeed, $(\fS_{L_j}, E, \N)^a$ is an object of $(X, M_X)_{\Prism}$ via $\iota_j : \fS_R / (E(u)) \cong R \rightarrow \cO_{L_j}$. This moreover induces the map of their coproducts $\iota_j^{(i)} : (\fS_R^{(i)}, E, M_{\mathrm{Spf}(\fS_R^{(i)})})^a \rightarrow (\fS_{L_j}^{(i)}, E, \N^{i+1})^a$ for $i=1,2$ that commutes with projection maps in $(X, M_X)_{\Prism}^{\text{opp}}$. Similar to the case for $\iota_n$, functoriality of the equivalence in Thm. \ref{cat-eq-rel} implies that for each $ 1 \leq j \leq r$ the descent data attached to the pullback $\iota_j^\star \bL$ of $\bL$ along $ \iota_j : \mathrm{Spa}(L_j, \cO_{L_j}) \rightarrow \cX$ is given by $(\cM_{L_j}, \varphi_{\cM_j}, f_j)$ in $\textbf{DD}_{\cO_{\cE, L_j}}$; where $\cM_{L_j}=\cM_R \otimes_{\fS_R[E^{-1}]^\wedge_p, \iota_j} \fS_{L_j}[E^{-1}]^\wedge_p$ and $f_{\text{ét},j}=\mathrm{id}_{\fS_{L_j}^{(1)}[E^{-1}]^\wedge_p} \otimes_{\fS_R^{(1)}[E^{-1}]^\wedge_p,\ \iota_{j}^{(1)}} f_{\text{ét}}$ is isomorphism $$f_{\text{ét},j} :\fS_{L_j}^{(1)}[E^{-1}]^\wedge_p \otimes_{p_1, \fS_{L_j}[E^{-1}]^\wedge_p} \cM_{L_j} \rightarrow \fS_{L_j}^{(1)}[E^{-1}]^\wedge_p \otimes_{p_2, \fS_{L_j}[E^{-1}]^\wedge_p} \cM_{L_j}$$ compatible with Frobenius and satisfies cocycle conditions over $\fS_{L_j}^{(2)}[E^{-1}]^\wedge_p$ for each $1 \leq j \leq r$. In particular, if $\bL$ is a finite $E$-height $\Z_p$-local system over $\cX$ with the associated descent data $(\fM_R, \varphi_{\fM_R}, f_{\text{ét}})$ then the descent data corresponding to the pullback $\iota_j^\star \bL$ realized as a $\Z_p$- representation of $\Gal_{L_j}$ (by Thm. 4.1 in \cite{DLMSII}) is of finite $E$-height and its descent data is given by the triplet $(\fM_{L_j}, \varphi_{\fM_{L_j}}, f_{\text{ét},j})$; where $\fM_{L_j} : = \fS_{L_j} \otimes_{\iota_j, \fS_{R}} \fM_R$ for each $1 \leq j \leq r$. Moreover, the converse of this also holds which can be regarded as a purity result for finite $E$-height $\Z_p$-local system.\smallskip

\begin{prop}\label{purity-fin-height}
If $\bL$ is a $\Z_p$-local system over $\cX$ such that for each $1 \leq j \leq r$ the pullback $\iota_j^\star \bL$ as a $\Z_p$- representation of $\Gal_{L_j}$ is of finite $E(u)$-height then $\bL$ is a finite $E$-height $\Z_p$-local system over $\cX$. 
\end{prop}

\begin{proof}
Let $(\cM_R, \varphi_{\cM_R}, f_{\text{ét}})$ be the descent data associated to $\bL$ over $\cX$. For each $1 \leq j \leq r$, we have the finite height descent data $(\fM_{L_j}, \varphi_{\fM_{L_j}}, f_{\text{ét},j})$ with $\fM_{L_j} : = \fS_{L_j} \otimes_{\iota_j, \fS_{R}} \fM_R$. Consider $$\fM_R := \cap_{j=1}^{r} \cM_R \cap \fM_{L_j}\ \text{equipped with induced Frobenius}\ \varphi_{\fM}.$$ This is a $p$-adically complete $\fS_R$-module as $\cM_R$ and $\fM_{L_j}$ are so. Prop. 5.7 in \cite{DLMSII} implies that $\fM_R$ is of finite $E$-height. Using Beauville-Laszlo theorem, Prop. 5.9, Prop. 5.13 and Prop. 5.14 in \cite{DLMSII} we conclude that $(\fM_R, \varphi_{\fM_R}, f_{\text{ét}}) \in \textbf{DD}_{\cO_{\cE}}^{\text{fin}}$ and $\cM_R = \cO_{\cE, R} \otimes_{\fS_R} \fM_R$. This completes the proof. 
\end{proof}

\noindent Note that our definition of finite $E(u)$-height for a $\Z_p$-local system on $\cX$ is given in terms of a semistable model $X=\mathrm{Spf}(R)$. However, the above proposition immediately implies that the notion of finite $E(u)$-height for a $\Z_p$-local system is independent of the choice of the integral model. This replicates the argument in Cor. 5.5 in \cite{DLMSII}. The points $\{ \xi_1, \xi_2, \cdots, \xi_r \}$ in $X$ gives rise to certain rank one points in $\cX$ (namely, $\mathrm{Spa}(L_j, \cO_{L_j}) \rightarrow \cX$), which are called $X$-Shilov points (see Def. 5.3 and Rem. 5.4 in \cite{DLMSII}). Now, suppose $\bL$ is of finite $E(u)$-height $\Z_p$-local system over $\cX$. It is clear by functoriality that if $x$ is a rank one point in $\cX$ the pullback $x^\star \bL$, regarded as $\Z_p$-representation of $\Gal_{k(x)}$, is of finite $E(u)$-height. In particular, the same holds for each $X$-Shilov points. The converse holds due to Prop. \ref{purity-fin-height}. Therefore, we conclude that $\bL$ is a finite $E(u)$-height $\Z_p$-local system if and only if its pullback along each rank one point of $\cX$ as $\Z_p$-representation over CDVR with imperfect residue field is of finite $E(u)$-height. 


\subsection{Proof of Main Thm. \ref{main-thm-1}:}\label{sec-5-3}

We are now ready to prove the Thm. \ref{main-thm-1}. The proof relies on Thm. \ref{mainthm-cdvf} and the purity Thm. \ref{purity-thm}. The setup is given by $\cX=\mathrm{Spa}(R[1/p],R)$, $\cX_n=\mathrm{Spa}(R_n[1/p],R_n)$ and finite étale map $\iota_n : \cX_n \rightarrow \cX$ as in previous Subsec. \ref{sec-5-2}. \smallskip

\begin{thm}
Assume that $m: = \text{max} \{ i \in \N \cup \{0\}: \zeta_{p^i} \in K^{un} \}$. Let $\bL$ be a finite $E$-height $\Z_p$-local system over $\cX$. Then its pullback $\bL_m:=\iota_m^\star \bL$ of $\bL$ along the finite étale cover $\iota_m : \cX_m \rightarrow \cX$ is semistable $\Z_p$-local system over $\cX_m$. 
\end{thm}

\begin{proof}

\noindent The finite étale map $\iota_m : \cX_m \rightarrow \cX$ is induced by the map of their integral models $\iota_m : \mathrm{Spf}(X_m) \rightarrow \mathrm{Spf}(X)$ (\textit{not} étale) and it induces a bijection between the respective Shilov points. For each $1 \leq j \leq r$, we have the CDVR $$(R_m)^\wedge_{(X_j)} =\cO_{L'_j}= \cO_{L_j}[\pi_m, X_1^{1/p^m}, X_2^{1/p^m}, \cdots, X_{j-1}^{1/p^m}, X_{j+1}^{1/p^m}, \cdots, X_r^{1/p^m}].$$ We write $L'_j=\mathrm{Frac}(\cO_{L'_j})=L_j(\pi_m, X_1^{1/p^m}, X_2^{1/p^m}, \cdots, X_{j-1}^{1/p^m}, X_{j+1}^{1/p^m}, \cdots, X_r^{1/p^m})$. Note that $L'_j$ is a finite extension of $L_j$. Therefore, the $X_m$-Shilov points of $\cX_m$ are given by the rank one points $x_j' : \mathrm{Spa}(L'_j, \cO_{L'_j}) \rightarrow \cX_m$.\smallskip

\noindent The $X$-Shilov points of $\cX$ are $x_j : \mathrm{Spa}(L_j, \cO_{L_j}) \rightarrow \cX$ for $ 1 \leq j \leq r$. For each $j \in \{ 1,2, \cdots, r \}$, the $\Z_p$-representation $x_j^\star \bL$ of $\Gal_{L_j}$ is of finite $E(u)$-height. Thm. \ref{mainthm-cdvf} implies that the restriction $x_j^\star \bL|_{\Gal_{L_j(\pi_m)}}$ is a semistable $\Z_p$-representation of $\Gal_{L_j(\pi_m)}$. Its further restriction $x_j^\star \bL|_{\Gal_{L'_j}}$ is a semistable $\Z_p$-representation of $\Gal_{L'_j}$ for each $1 \leq j \leq r$.\smallskip

\noindent Now we use functoriality to see that the pullback $(x'_j)^\star \bL_m$ of $\bL_m$ along the $X_m$-Shilov points $x_j' : \mathrm{Spa}(L'_j, \cO_{L'_j}) \rightarrow \cX_m$ is isomorphic to $x_j^\star \bL|_{\Gal_{L'_j}}$ as $\Z_p$-representation of $\Gal_{L'_j}$. In particular, $(x'_j)^\star \bL_m$ is a semistable $\Z_p$-Galois representation of $\Gal_{L'_j}$ for all $1 \leq j \leq r$. Hence, the purity result i.e. Thm. \ref{purity-thm} implies that $\bL_m$ is a semistable $\Z_p$-local system of $\cX_m$. 

\end{proof}

\end{document}